\documentclass[10pt,twoside,english,reqno,a4paper]{amsart}

\usepackage{listings,graphicx,amsmath,varioref,amscd,amssymb,color,bm,stmaryrd,amsthm,amsfonts,graphics,geometry,latexsym,pgf,pst-all} 
\theoremstyle{plain}
\usepackage{esint}
\usepackage{amsthm}

\usepackage{todonotes,mathtools,comment}
\usepackage{enumerate}
\usepackage{amsrefs}

\theoremstyle{plain}
\newtheorem{theorem}{Theorem}[section]
\newtheorem{proposition}[theorem]{Proposition}
\newtheorem{lemma}[theorem]{Lemma}

\usepackage{geometry}
\usepackage[colorlinks=false]{hyperref}

\theoremstyle{definition}
\newtheorem{definition}[theorem]{Definition}

\newtheorem{remark}[theorem]{Remark}

\theoremstyle{remark}

\usepackage{fouriernc}
\usepackage[T1]{fontenc}

\numberwithin{equation}{section}

\newcommand{\N}{\mathbb{N}}
\newcommand{\R}{\mathbb{R}}

\def\sob{W^{1,p}_{0}(\Omega)}

\def\luo{L^{1}(\Omega)}
\def\lio{{L^{\infty}(\Omega)}}

\def\co{{C(\overline\Omega)}}
\def\cuo{{C^1_0(\overline\Omega)}}

\def\into{\int_{\Omega}}

\def\ae{\mathrm{a.e.}}

\def\uhgamma{{\underline{h}_{\gamma_1,\gamma_2}(L)}}

\def\ohgamma{{\overline{h}_{\gamma_1,\gamma_2}(L)}}

\begin{document}
\title[Regularizing effect of the natural growth term in quasilinear problems]{Regularizing effect of the natural growth term in quasilinear problems with sign-changing nonlinearities}

\author[J. Carmona Tapia]{Jos\'{e} Carmona Tapia}
\address[Jos\'{e} Carmona Tapia]{Departamento de Matem\'aticas\\ Uni\-ver\-si\-dad de Alme\-r\'ia\\Ctra. Sacramento s/n\\
La Ca\-\~{n}a\-da de San Urbano\\ 04120 - Al\-me\-r\'{\i}a, Spain}
\email{jcarmona@ual.es}

\author[P. Malanchini]{Paolo Malanchini}
\address[Paolo Malanchini]{Dipartimento di Matematica e Applicazioni\\ Universit\`a degli Studi di Milano - Bicocca, via Roberto Cozzi 55, 20125 - Milano, Italy}
\email{p.malanchini@campus.unimib.it}

\author[A. J. Mart\'{i}nez Aparicio]{Antonio J. Mart\'{i}nez Aparicio}
\address[Antonio J. Mart\'{i}nez Aparicio]{Departamento de Matem\'aticas\\ Uni\-ver\-si\-dad de Alme\-r\'ia\\Ctra. Sacramento s/n\\
La Ca\~{n}ada de San Urbano\\ 04120 - Al\-me\-r\'{\i}a, Spain}
\email{ama194@ual.es}

\author[P. J. Mart\'{i}nez-Aparicio]{Pedro J. Mart\'{i}nez-Aparicio}
\address[Pedro J. Mart\'{i}nez-Aparicio]{Departamento de Matem\'aticas\\ Uni\-ver\-si\-dad de Alme\-r\'ia\\Ctra. Sacramento s/n\\
La Ca\~{n}ada de San Urbano\\ 04120 - Al\-me\-r\'{\i}a, Spain}
\email{pedroj.ma@ual.es}


\begin{abstract}

We investigate the existence and nonexistence of solutions to the Dirichlet problem
\begin{equation*}
\tag{$P$}
\label{eq:PbAbstract}
\left\{ \begin{alignedat}{2}
-\Delta_p u + g(u) |\nabla u|^p &= \lambda f(u) \quad &&\mbox{in} \;\; \Omega, \\
u &= 0 \quad &&\mbox{on} \;\; \partial\Omega, 
\end{alignedat}
\right.
\end{equation*}
where $\Omega\subset \R^N$ is a smooth bounded domain, $p\in (1,\infty)$, $\lambda>0$ and $g\in C(\R)$. Our main assumption is that $f\colon \R\to \R$ is a continuous function such that $f(s)>0$ for all $s\in (\alpha,\beta)$, where $0<\alpha<\beta$ are two zeros of $f$. 

If $f(0)\geq 0$, we show that an area condition involving $f$ and $g$ is both sufficient and necessary in order to have a pair $(\lambda,u)\in \R^+\times \cuo$, with $u\geq 0$ and $\|u\|_\co\in (\alpha,\beta]$, solving~\eqref{eq:PbAbstract}. 

We also study how the presence of the gradient term affects the existence of solution. Roughly speaking, the more negative $g$ is, the stronger its regularizing effect on~\eqref{eq:PbAbstract}. We prove that, regardless of the shape of $f$, for any fixed $\lambda$, there always exists a function $g$ such that~\eqref{eq:PbAbstract} admits a nonnegative solution with maximum in $(\alpha,\beta]$.
\end{abstract}

\maketitle
 
\tableofcontents

\section{Introduction}
In this article, we study the quasilinear problem
\begin{equation}
\label{eq:PbCV}
\left\{ \begin{alignedat}{2}
-\Delta_p u + g(u) |\nabla u|^p &= \lambda f(u) \quad &&\mbox{in} \;\; \Omega, \\
u &= 0 \quad &&\mbox{on} \;\; \partial\Omega, 
\end{alignedat}
\right.
\end{equation}
where $\Omega\subset \R^N$ ($N\geq 1$) is a smooth bounded domain, $p\in (1,\infty)$, $\lambda>0$ is a parameter and $g\colon \R\to \R$ is just a continuous function. Moreover, we suppose that $f\colon \R\to \R$ is continuous, has two consecutive positive zeros, and is positive between them. We denote these two zeros by $\alpha$ and $\beta$ with $0<\alpha<\beta$. Then, we are assuming that $f(\alpha)=f(\beta)=0$ and $f(s)>0$ for all $s\in(\alpha,\beta)$. We underline that $f$ is allowed to change sign outside the interval $[\alpha,\beta]$.

Under these hypotheses, problem~\eqref{eq:PbCV} has been widely studied when the natural growth term is not present, i.e., when $g\equiv 0$. The study of problems like
\begin{equation}
\label{eq:PbSemi}
\left\{ \begin{alignedat}{2}
-\Delta_p u  &= \lambda f(u) \quad &&\mbox{in} \;\; \Omega, \\
u &= 0 \quad &&\mbox{on} \;\; \partial\Omega,
\end{alignedat}
\right.
\end{equation}
goes back to~\cite{Hess}. For the usual Laplacian operator ($p=2$) and assuming $f(0)\geq 0$, the author proved that an area condition on $f$ is sufficient to guarantee the existence of two nonnegative solutions to~\eqref{eq:PbSemi} with maximum in $(\alpha,\beta]$ for large $\lambda$. Specifically, this area condition is
\begin{equation}
\label{eq:area_cond_semi}
    \int_{s}^\beta f(\eta) \ \mathrm d \eta>0,\ \forall s\in [0,\beta).
\end{equation}
Later, in~\cite{Cl-Sw} and~\cite{Dan-Sch}, the authors showed using different techniques that~\eqref{eq:area_cond_semi} is also necessary in order to have, for some $\lambda>0$, a nonnegative solution $u\in C^2(\overline\Omega)$ to~\eqref{eq:PbSemi} (with $p=2$) such that $\|u\|_\co\in [\alpha,\beta]$.

For the $p$-Laplacian, these results have been suitably extended. Two different proofs of the sufficiency and necessity of~\eqref{eq:area_cond} for the existence of solutions to~\eqref{eq:PbSemi} with maximum in $[\alpha,\beta]$ can be found in~\cite{Guo1} and~\cite{Ng-Sch}. Beyond existence, some qualitative properties of these solutions have also been studied, mainly in the case $p=2$. Regarding the behaviour of the solutions as $\lambda$ goes to infinity, we refer the reader to~\cite{Barrios, GuoWebb, Dan-Wei, Jang} when $f$ is regular (say $C^1$), and to~\cite{Guo2, Guo3} when $f$ is not differentiable at $\alpha$ or $\beta$. The specific structure of the solution set of~\eqref{eq:PbCV} is investigated in~\cite{Senos, Wei}.

In recent years, this kind of Dirichlet problem in bounded domains with sign-changing nonlinearities has been widely studied for other operators. Among the operators considered are the $p(x)$-Laplacian (\cite{Ho-Kim}), the $\phi$-Laplacian (\cite{Correa}), the 1-Laplacian (\cite{1-Lap}), the fractional Laplacian (\cite{Car-Fi}), a Kirchhoff operator (\cite{Arc-Car-MaAp}) and a Schr\"o\-din\-ger-type operator (\cite{dS-Sil}). In all these works, condition~\eqref{eq:area_cond_semi} is assumed to establish the existence of nonnegative solutions with maximum in $[\alpha,\beta]$; however, the necessity of~\eqref{eq:area_cond_semi} has been less explored. We point out that the situation changes under Robin boundary conditions; in that case, no area condition is required, as shown in~\cite{Robin}.

An original contribution of the present work is the derivation of an area condition for problems of the form~\eqref{eq:PbCV}, extending the classical condition~\eqref{eq:area_cond_semi}. The condition for problem~\eqref{eq:PbCV} reflects an interaction between the nonlinearities $f$ and $g$ that has not been previously characterized which, when $g\not\equiv 0$, completely differs from~\eqref{eq:area_cond_semi}. This highlights new structural properties in Dirichlet problems involving nonlinearities with multiple zeros. Concretely, we prove that the area condition for~\eqref{eq:PbCV} is
\begin{equation}
\label{eq:area_cond}
    \int_{s}^\beta f(\eta) e^{-\frac{p}{p-1}G(\eta)} \ \mathrm d \eta>0,\ \forall s\in [0,\beta),
\end{equation}
where $G(s)\coloneq \int_0^s g(\eta) \ \mathrm{d}\eta$ for all $s\in \R$. A major difficulty in dealing with~\eqref{eq:PbCV}, unlike the operators mentioned above, is the absence of a variational structure. We overcome this issue by transforming~\eqref{eq:PbCV} into an equivalent variational problem and show that~\eqref{eq:area_cond} is sufficient and necessary for the existence of a pair $(\lambda, u)$, with $u\geq 0$ and $\|u\|_\co\in [\alpha,\beta]$, solving~\eqref{eq:PbCV}.

Our first result is as follows.

\begin{theorem}
\label{th:existence}
Let $f\in C(\R)$ satisfy $f(s)>0$ for $s\in (\alpha,\beta)$, where $0<\alpha<\beta$ are two zeros of $f$, and let $g\in C(\R)$. Then the following holds:
\begin{enumerate}[i)]
    \item If $f$ verifies~\eqref{eq:area_cond} and $f(0)\geq 0$, then there is some $\overline{\lambda}>0$ such that, for every $\lambda> \overline{\lambda}$, problem~\eqref{eq:PbCV} has a nonnegative solution $u\in\cuo$ with $\|u\|_\co \in (\alpha,\beta]$. 
    \item If $f$ does not satisfy~\eqref{eq:area_cond}, then problem~\eqref{eq:PbCV} admits no nonnegative solution with maximum in $[\alpha,\beta]$ for any $\lambda>0$.
\end{enumerate}
\end{theorem}

A significant consequence of this work is that, in certain cases, condition~\eqref{eq:area_cond} imposes a less restrictive constraint on $f$ than the classical condition~\eqref{eq:area_cond_semi}. Therefore, even when problem~\eqref{eq:PbSemi} admits no nonnegative solutions with maxima in $[\alpha,\beta]$ because~\eqref{eq:area_cond_semi} is not satisfied, the gradient term in~\eqref{eq:PbCV} may induce a regularizing effect, allowing the existence of such solutions for~\eqref{eq:PbCV} nonetheless. This establishes a novel mechanism by which the gradient term expands the solvability framework of problem~\eqref{eq:PbSemi}. To better illustrate this phenomenon, we include a parameter multiplying the natural growth term and we analyse the existence and nonexistence of solution when the parameter varies. The family of problems considered is
\begin{equation}
\label{eq:PbL}
\left\{ \begin{alignedat}{2}
-\Delta_p u + L g(u) |\nabla u|^p &= \lambda f(u) \quad &&\mbox{in} \;\; \Omega, \\
u &= 0 \quad &&\mbox{on} \;\; \partial\Omega, 
\end{alignedat}
\right.
\end{equation}
where $L\in \R$. In addition, $g$ is assumed to have constant sign. Since $L\in \R$, without loss of generality, we may assume that $g$ is positive.  

We begin by studying the range of $L$'s for which a solution $u\in\cuo$ to~\eqref{eq:PbL} with $u\geq 0$ and $\|u\|_\co\in [\alpha,\beta]$ exists for some $\lambda>0$. To simplify the notation, we define the set
\begin{equation}
    \label{eq:def_Lfrak}
    \mathfrak{L} \coloneqq \left\{L\in \R: \exists \lambda>0 \text{ s.t. } \eqref{eq:PbL} \text{ has a solution } 0\leq u\in \cuo \text{ with } \|u\|_\co\in [\alpha,\beta] \right\}.
\end{equation}
As we prove, this set, which is closely related to the area condition, is nonempty for any given $f$, regardless of its shape. Furthermore, for $L\in\mathfrak{L}$, we also analyse the behaviour as $L$ varies of the quantities
\begin{align}
    \label{eq:def_lambda_min}
    \lambda_{\rm min}(L) &\coloneqq \inf \left\{\lambda\geq 0: \eqref{eq:PbL} \text{ has a solution } 0\leq u\in \cuo \text{ with } \|u\|_\co\in [\alpha,\beta] \right\}, \\
    \label{eq:def_olambda_min}
    \overline{\lambda}_{\rm min}(L) &\coloneqq \inf \left\{\overline{\lambda}\geq 0: \eqref{eq:PbL} \text{ has a solution } 0\leq u\in \cuo \text{ with } \|u\|_\co\in [\alpha,\beta] \text{ for any } \lambda > \overline{\lambda} \right\}.
\end{align}
Observe that, by definition, $\lambda_{\rm min}(L) \leq \overline{\lambda}_{\rm min}(L)$.

The next result asserts that, regardless of the behaviour of $f$ outside the interval $[\alpha,\beta]$, there always exists a parameter $L\in \R$ such that problem~\eqref{eq:PbL} admits a solution with its maximum attained in $[\alpha,\beta]$ for sufficiently large values of $\lambda$. Hence, the presence of the natural growth term induces a regularizing effect on problem~\eqref{eq:PbCV}. In heuristic terms, the magnitude and sign of the function $g$ modulate this effect. Roughly speaking, the more negative $g$ is, the stronger the regularizing effect becomes. Conversely, if $g$ is sufficiently positive, a nonexistence phenomenon arises and no solution with maximum in $[\alpha,\beta]$ exists for any $\lambda>0$.

\begin{theorem}
\label{th:Lfrak}
Let $f\in C(\R)$ satisfy $f(0)\geq 0$ and $f(s)>0$ for $s\in (\alpha,\beta)$, where $0<\alpha<\beta$ are two zeros of $f$, and let $g\in C(\R)$ be positive. Then there exists $\tilde L\in (-\infty,\infty]$ such that $\mathfrak{L}$, defined in~\eqref{eq:def_Lfrak}, verifies
\[
\mathfrak{L} = (-\infty,\tilde L),
\]
where $\tilde L = \infty$ if $f\geq 0$ and $\tilde L<\infty$ if $f$ changes sign in $[0,\beta]$.

Furthermore, for $L<\tilde L$, the quantities $\lambda_{\rm min}(L)$ and $\overline{\lambda}_{\rm min}(L)$ defined in~\eqref{eq:def_lambda_min} and~\eqref{eq:def_olambda_min} are positive and satisfy the following:
\begin{enumerate}[i)]
    \item $\overline{\lambda}_{\rm min}(L)\to 0$ as $L\to -\infty$,
    \item $\lambda_{\rm min}(L)\to \infty$ as $L\to \tilde L$.
\end{enumerate}
\end{theorem}

An important corollary of Theorem~\ref{th:Lfrak} is that, given any fixed $\lambda$, there are two numbers $L_1<L_2$ such that, if $L<L_1$, problem~\eqref{eq:PbL} admits a solution $0\leq u\in\cuo$ to~\eqref{eq:PbL} with $\|u\|_\co\in [\alpha,\beta]$, whereas if $L>L_2$, no such solution exists for~\eqref{eq:PbL}. This allows us to prove, for any fixed $\lambda$, the existence of a maximal solution among the solutions contained in the interval $[0,\beta]$ to problem~\eqref{eq:PbL} when $L$ is negative enough. In the following, we investigate the behaviour of these maximal solutions as $L\to -\infty$.

\begin{theorem}
\label{th:conv_beta}
Let $f\in C(\R)$ satisfy $f(0)\geq 0$ and $f(s)>0$ for $s\in (\alpha,\beta)$, where $0<\alpha<\beta$ are two zeros of $f$, and let $g\in C(\R)$ be positive. Given $\lambda>0$, there is some $L_\lambda\in \R$ such that the maximal solution $\overline u_L$ in the interval $[0,\beta]$ to problem~\eqref{eq:PbL} exists for every $L\leq L_\lambda$, and
\[
\|\overline{u}_L\|_{\co} \to \beta \text{ as } L\to -\infty.
\]
\end{theorem}

We point out that all our results can be suitably extended to the more general problem
\begin{equation}
\label{eq:PbExtension3}
\left\{ \begin{alignedat}{2}
-\operatorname{div} (a(u) |\nabla u|^{p-2} \nabla u) + g(u) |\nabla u|^p &= \lambda f(u) \quad &&\mbox{in} \;\; \Omega, \\
u &= 0 \quad &&\mbox{on} \;\; \partial\Omega,
\end{alignedat}
\right.
\end{equation}
where $a\in C^1(\R)$ is positive, under the area condition
\begin{equation}
\label{eq:ac_extension_3}
    \int_{s}^\beta f(\eta) a(\eta)^\frac{1}{p-1} e^{-\frac{p}{p-1}\int_0^\eta \frac{g(\sigma)}{a(\sigma)} \, \mathrm d\sigma} \ \mathrm d \eta > 0,\ \forall s\in [0,\beta).
\end{equation}
When $g=\frac{1}{p}a'$, the problem has a variational structure and~\eqref{eq:ac_extension_3} reduces to the usual area condition~\eqref{eq:area_cond_semi}. This is the case, for instance, of the Schr\"odinger operator (cf.~\cite{Schrod-1}). We further discuss this topic in Section~\ref{sec:equivalent}.

The plan of the paper is the following. In Section~\ref{sec:initial}, we present the approach adopted throughout the article and prove Theorem~\ref{th:existence}. By means of a detailed analysis of how the area condition varies with $L$ in~\eqref{eq:PbL}, in Section~\ref{sec:regularizing} we prove Theorem~\ref{th:Lfrak} and Theorem~\ref{th:conv_beta}. Finally, in Section~\ref{sec:further}, we extend our results to certain generalizations of~\eqref{eq:PbCV} and provide further insights. In particular, we study problem~\eqref{eq:PbExtension3}, which includes a more general divergence term, and we analyse~\eqref{eq:PbCV} in the case where $g(s)$ depends also on $x$. Moreover, we discuss some regularity conditions on $f$ which ensure that the solutions have norm different from~$\beta$. We also show that a regularizing effect similar to the one observed in~\eqref{eq:PbL} as $L\to -\infty$ can also be obtained in~\eqref{eq:PbCV} if $p\to 1^+$.

\section{An initial approach}
\label{sec:initial}

Throughout this work, we always deal with bounded solutions. In this setting, solutions to~\eqref{eq:PbCV} enjoy good regularity properties. The notion of solution we use is the following.

\begin{definition}
    A function $u\in C^1(\overline{\Omega})$ is a \textit{subsolution} (resp. a \textit{supersolution}) to~\eqref{eq:PbCV} if it verifies
    \begin{equation*}
    \into |\nabla u|^{p-2}\nabla u \nabla \varphi + \into g(u)|\nabla u|^p \varphi \overset{(\geq)}{\leq} \lambda\into f(u)\varphi,\ \forall \varphi\in \sob\cap \lio \text{ with } \varphi\geq 0,
    \end{equation*}
    and $u \leq 0$ (resp. $u\geq 0$) on $\partial\Omega$. In the same way, we say that $u\in \cuo$ is a \textit{solution} to~\eqref{eq:PbCV} if
    \begin{equation}
    \label{eq:def_sol}
        \into |\nabla u|^{p-2}\nabla u \nabla \varphi + \into g(u)|\nabla u|^p \varphi = \lambda\into f(u)\varphi,\ \forall \varphi\in \sob\cap \lio.
    \end{equation}
\end{definition}

\begin{remark}
    Thanks to the structure of problem~\eqref{eq:PbCV} and to the boundedness of the solutions, the usual weak solution concept, where solutions belong to the Sobolev space $\sob$, is equivalent to our definition. Indeed, any bounded $\sob$-solution to~\eqref{eq:PbCV} belongs to $C^{1,\mu}(\overline{\Omega})$, where $\mu\in (0,1)$ depends on the specific shape of~\eqref{eq:PbCV} (see~\cite{Lieb}).

    We stress that, in general, due to the degeneracy of the $p$-Laplacian operator, solutions are not expected to belong to $C^2(\Omega)$ even if $g\equiv 0$ (see~\cite{Sciunzi}). This is a major difference with respect to the usual Laplacian operator.
\end{remark}

In this section, our main aim is to study existence and nonexistence of solutions to problem~\eqref{eq:PbCV}. Our strategy is to transform~\eqref{eq:PbCV} into a semilinear problem via a suitable change of variables. To this end, let $\Psi\colon \R\to\R$ be the function given by
\begin{equation}
    \label{eq:def_Psi}
     \Psi(s) \coloneqq \int_0^s e^{-\frac{1}{p-1} G(\eta)} \ \mathrm{d}\eta, \ \forall s\in \R,
\end{equation}
where $G(s) \coloneq \int_0^s g(\eta)\ \mathrm{d}\eta$ for any $s\in\R$. We stress that $\Psi$ is an increasing $C^2$ function whose derivative is
\[
\Psi'(s)= e^{-\frac{1}{p-1} G(s)},\ \forall s\in\R.
\]
Its inverse, whose domain is $\operatorname{Dom}\big(\Psi^{-1} \big) = \big( \lim_{s\to -\infty} \Psi(s), \lim_{s\to\infty} \Psi(s) \big)$, is increasing and belongs to $C^2$. Indeed, its derivative is
\[
\left(\Psi^{-1}\right)'(s) = \frac{1}{\Psi'(\Psi^{-1}(s))} = e^{\frac{1}{p-1} G(\Psi^{-1}(s))}, \ \forall s\in \operatorname{Dom}\big(\Psi^{-1} \big).
\]
Observe that if $u$ is a solution to~\eqref{eq:PbCV}, then $v=\Psi(u)$ is formally a solution to the problem
\begin{equation}
\label{eq:PbSemiCV}
\left\{ \begin{alignedat}{2}
-\Delta_p v  &= \lambda \tilde f(v) \quad &&\mbox{in} \;\; \Omega, \\
v &= 0 \quad &&\mbox{on} \;\; \partial\Omega,
\end{alignedat}
\right.
\end{equation}
where $\tilde f\colon \operatorname{Dom}\big(\Psi^{-1} \big) \to \R$ is defined as
\begin{equation}
    \label{eq:def_tilde_f}
    \tilde{f}(s) \coloneqq f\left(\Psi^{-1}(s) \right) e^{-G(\Psi^{-1}(s))},\ \forall s\in \operatorname{Dom}\big(\Psi^{-1} \big).
\end{equation}
In the next result, we formalise this equivalence.

\begin{proposition}
\label{prop:equiv}
    Let $f, g\in C(\R)$. Then $u\in \cuo$ is a solution to~\eqref{eq:PbCV} if and only if $v=\Psi(u) \in \cuo$ is a solution to~\eqref{eq:PbSemiCV}.
\end{proposition}

\begin{remark}
    Since $\Psi$ is increasing and $\Psi(0)=0$, let us note that $u\geq 0$ if and only if $v\geq 0$.
\end{remark}

\begin{proof}
Let $u\in \cuo$ be a solution to~\eqref{eq:PbCV}. Given $\phi\in \sob\cap\lio$, we take $e^{-G(u)} \phi \in \sob\cap\lio$ as test function in~\eqref{eq:def_sol}. After cancelling terms, we obtain that
\begin{equation*}
    \into e^{-G(u)} |\nabla u|^{p-2} \nabla u \nabla \phi = \lambda \into f(u) e^{-G(u)} \phi, \ \forall \phi\in \sob\cap\lio.
\end{equation*}
If we define $v\coloneq \Psi(u)$, this can be rewritten as 
\begin{equation*}
    \into |\nabla v|^{p-2} \nabla v \nabla \phi = \lambda \into \tilde f(v) \phi, \ \forall \phi\in \sob\cap\lio.
\end{equation*}
Then, $v=\Psi(u)$ is a solution to~\eqref{eq:PbSemiCV}.

Reciprocally, let $v \in \cuo$ be a solution to~\eqref{eq:PbSemiCV} contained in $\operatorname{Im}(\Psi)$. To show that $u \coloneqq \Psi^{-1}(v)$ is a solution to~\eqref{eq:PbCV}, given $\varphi\in \sob\cap\lio$, we take $e^{G(u)} \varphi \in \sob\cap\lio$ as test function in~\eqref{eq:PbSemiCV}. Rewriting the integrals in terms of $u$, we deduce that
\begin{align*}
    \into e^{G(u)} |\nabla \Psi(u)|^{p-2} \nabla \Psi (u) \nabla \varphi
+ \into g(u) e^{G(u)} \nabla u |\nabla \Psi(u)|^{p-2}\nabla \Psi(u) \, \varphi &\\ =  \lambda \into \tilde f(\Psi(u)) e^{G(u)} \varphi, \ \forall \varphi\in \sob\cap\lio.
\end{align*}
Taking into account the definition of $\tilde f$ (see~\eqref{eq:def_tilde_f}) and using that $|\nabla \Psi(u)|^{p-2}\nabla \Psi(u) =  e^{-G(u)}|\nabla u|^{p-2}\nabla u$, we obtain that~\eqref{eq:def_sol} holds. In this way, we conclude that $u=\Psi^{-1}(v)$ is a solution to~\eqref{eq:PbCV}.
\end{proof}

In the following, we prove Theorem~\ref{th:existence}. The strategy is to obtain, using Proposition~\ref{prop:equiv}, the existence and nonexistence results for problem~\eqref{eq:PbCV} from those already known for problem~\eqref{eq:PbSemiCV}.

\begin{proof}[Proof of Theorem~\ref{th:existence}]
Under our assumptions on $f$ and $g$, the function $\tilde f$, defined in~\eqref{eq:def_tilde_f}, has two consecutive zeros, $\Psi(\alpha)$ and $\Psi(\beta)$, and is positive between them. Moreover, $f(0)\geq 0$ if and only if $\tilde f(0)\geq 0$.

Thanks to~\cite{Ng-Sch} (see also~\cite[Section~4]{Guo1}), we know that the existence of nonnegative solutions to~\eqref{eq:PbSemiCV} with $\co$-norm between $\Psi(\alpha)$ and $\Psi(\beta)$ is closely related to the area condition
\begin{equation}
    \label{eq:Pf_Th_ex_1}
    \int_{s}^{\Psi(\beta)} \tilde f(\eta) \ \mathrm d \eta>0,\ \forall s\in [0,\Psi(\beta)).
\end{equation}
Indeed, when~\eqref{eq:Pf_Th_ex_1} holds and $\tilde f(0)\geq 0$,~\cite[Theorem~1.1]{Ng-Sch} ensures the existence of some $\overline{\lambda}$ such that, for any $\lambda>\overline{\lambda}$, problem~\eqref{eq:PbSemiCV} has a nonnegative solution $v\in\cuo$ with $\|v\|_\co\in [\Psi(\alpha), \Psi(\beta)]$. In contrast, when~\eqref{eq:Pf_Th_ex_1} is not verified, problem~\eqref{eq:PbSemiCV} has no nonnegative solutions with maximum between $\Psi(\alpha)$ and $\Psi(\beta)$ for any $\lambda>0$ (see~\cite[Theorem~3.1]{Ng-Sch}).

Now, observe that the area condition~\eqref{eq:area_cond} imposed on $f$ is equivalent to~\eqref{eq:Pf_Th_ex_1}. In fact, performing the change of variables $\rho=\Psi^{-1}(\eta)$, one obtains
\begin{equation}
\label{eq:Pf_Th_ex_2}
\begin{split}
    \int_{s}^{\Psi(\beta)} \tilde f(\eta) \ \mathrm d \eta
    &= \int_{s}^{\Psi(\beta)} f\left(\Psi^{-1}(\eta) \right) e^{-G(\Psi^{-1}(\eta))} \ \mathrm{d}\eta = \int_{\Psi^{-1}(s)}^{\beta} f\left(\rho \right) \Psi'(\rho) e^{-G(\rho)} \ \mathrm{d}\rho\\
    &= \int_{\Psi^{-1}(s)}^{\beta} f\left(\rho \right)  e^{-{\frac{p}{p-1}} G(\rho)} \ \mathrm{d}\rho, \ \forall s\in \left[0, \Psi(\beta) \right).
\end{split}
\end{equation}

The proof ends just by applying Proposition~\ref{prop:equiv}.
\end{proof}

\begin{remark}
    Under certain regularity conditions on $f$, we can ensure that the maximum of any nonnegative solution to~\eqref{eq:PbCV} is different from $\beta$. To facilitate the reading, we leave this topic for Section~\ref{sec:smp}.
\end{remark}

\begin{remark}
    Imposing an additional assumption on the behaviour of $f$ at 0, it can be shown that every nonnegative (and nontrivial) solution to~\eqref{eq:PbCV} is positive. Concretely, this occurs whenever
    \begin{equation*}
        \liminf_{s\to 0^+} \frac{f(s)}{s^{p-1}} > -\infty.
    \end{equation*}
    This can be proved using the strong maximum principle (\cite{Vazquez}), following the same reasoning as in Section~\ref{sec:smp}.
\end{remark}

\begin{remark}
    When the existence conditions of Theorem~\ref{th:existence} hold, the existence of a maximal solution $\overline u_\lambda$ to~\eqref{eq:PbCV} in $[0,\beta]$ for each $\lambda >\overline{\lambda}$ is guaranteed because $\beta$ is a supersolution (see~\cite[Theorem~4.2]{BMP}). The variational arguments of~\cite[Lemma~2.1]{Cl-Sw} show that $\overline{v}_\lambda$, the maximal solution to~\eqref{eq:PbSemiCV} in $[0,\Psi(\beta)]$, verifies $\|\overline{v}_\lambda\|_\co \to \Psi(\beta)$ as $\lambda\to\infty$. Using Proposition~\ref{prop:equiv}, one deduces that $\|\overline{u}_\lambda\|_\co \to \beta$ as $\lambda\to\infty$.

    If $f$ is more regular and behaves well near the zeros, it is shown in~\cite[Theorem~2]{Cl-Sw} and~\cite[Theorem~A]{GuoWebb} that, for $\overline{v}_\lambda$, the convergence to $\Psi(\beta)$ is uniform in compact subsets of $\Omega$. Thanks to Proposition~\ref{prop:equiv}, it also holds that $\overline u_\lambda \to \beta$ uniformly in $K$ as $\lambda\to\infty$ for every compact set $K\subset \Omega$.
\end{remark}

\begin{remark}
    Under the existence conditions of Theorem~\ref{th:existence}, we can ensure the existence of two nonnegative solutions to~\eqref{eq:PbCV} with maximum in $(\alpha,\beta]$ for any $\lambda>\overline{\lambda}$. Again, this is a consequence of the multiplicity results known for~\eqref{eq:PbSemi} (see~\cite{Hess, Correa}), that also hold for~\eqref{eq:PbCV} thanks to Proposition~\ref{prop:equiv}.
\end{remark}

\section{Regularizing effect of the natural growth term}
\label{sec:regularizing}

Here, we study the effect of the interaction between the gradient term and $f$ on the existence and nonexistence of solutions. To this end, we introduce a parameter $L\in \R$ multiplying the gradient term, and we analyse how the solution set of problem~\eqref{eq:PbL}, i.e., of
\begin{equation}
\label{eq:PbL_Aux}
\left\{ \begin{alignedat}{2}
-\Delta_p u + L g(u) |\nabla u|^p &= \lambda f(u) \quad &&\mbox{in} \;\; \Omega, \\
u &= 0 \quad &&\mbox{on} \;\; \partial\Omega, 
\end{alignedat}
\right.
\end{equation}
varies as $L$ varies.

We divide this section into three parts. The first is more technical and concerns the behaviour of the solutions as $L\to -\infty$. In the second part, we present a useful stability result. Finally, we prove Theorem~\ref{th:Lfrak} and Theorem~\ref{th:conv_beta}.

\subsection{Asymptotic behaviour as \texorpdfstring{$L\to -\infty$}{L to -infinity}}

First of all, we prove an elementary lemma that will be useful for studying certain properties of the functions defined later in~\eqref{eq:def_Psi_L} and~\eqref{eq:def_hmin_hmax}.
\begin{lemma}
\label{lem:taylor}
    Let $G\in C^1(\R)$ be an increasing function and let $\gamma_1,\gamma_2\in \R$ be such that $\gamma_1<\gamma_2$. Then, one has
    \[
    \lim_{L\to -\infty} \left(\frac{e^{-LG(\gamma_2)}}{-LG'(\gamma_2)} \right)^{-1} \int_{\gamma_1}^{\gamma_2} e^{-LG(\eta)}\ \mathrm{d}\eta = 1.
    \]
\end{lemma}

\begin{proof}
Throughout the proof, we assume $L<0$. First, we use Taylor's Theorem to write
\begin{equation}
    \label{eq:L_taylor_1}
    G(s) = G(\gamma_2) - G'(\gamma_2) (\gamma_2-s) + o(\gamma_2-s), \ \forall s\in [\gamma_1,\gamma_2],
\end{equation}
where $o(\gamma_2-s)$ is such that $\frac{o(\gamma_2-s)}{\gamma_2-s} \to 0$ as $s\to \gamma_2$. Therefore, given $\varepsilon\in (0, G'(\gamma_2))$, there exists $\delta=\delta(\varepsilon)>0$ small such that
\begin{equation}
    \label{eq:L_taylor_2}
    \left|\frac{o(\gamma_2-s)}{\gamma_2-s}\right|<\varepsilon,\ \forall s\in (\gamma_2-\delta, \gamma_2).
\end{equation}
On the one hand, using the monotonicity of $G$, we deduce that 
\[
0 \leq \int_{\gamma_1}^{\gamma_2-\delta} e^{-LG(\eta)}\ \mathrm{d}\eta \leq (\gamma_2-\delta-\gamma_1) e^{-LG(\gamma_2-\delta)},\ \forall L<0.
\]
Since $G(\gamma_2-\delta)< G(\gamma_2)$, then $\lim_{L\to -\infty} L e^{L(G(\gamma_2)-G(\gamma_2-\delta))} = 0$. Hence, we obtain that
\begin{equation}
    \label{eq:L_taylor_3}
    \lim_{L\to -\infty} \left(\frac{e^{-LG(\gamma_2)}}{-LG'(\gamma_2)} \right)^{-1} \int_{\gamma_1}^{\gamma_2-\delta} e^{-LG(\eta)}\ \mathrm{d}\eta = 0.
\end{equation}
On the other hand, taking~\eqref{eq:L_taylor_2} into account, we have that $e^{-L(-\varepsilon)(\gamma_2-s)}  \leq e^{-Lo(\gamma_2-s)} \leq e^{-L\varepsilon(\gamma_2-s)}$ when $s\in (\gamma_2-\delta, \gamma_2)$ and, using~\eqref{eq:L_taylor_1}, we deduce
\begin{equation}
    \label{eq:L_taylor_4}
    e^{-LG(\gamma_2)} \int_{\gamma_2-\delta}^{\gamma_2} e^{L(G'(\gamma_2)+\varepsilon)(\gamma_2-\eta)}\ \mathrm{d}\eta \leq \int_{\gamma_2-\delta}^{\gamma_2} e^{-LG(\eta)}\ \mathrm{d}\eta \leq e^{-LG(\gamma_2)} \int_{\gamma_2-\delta}^{\gamma_2} e^{L(G'(\gamma_2)-\varepsilon)(\gamma_2-\eta)}\ \mathrm{d}\eta.
\end{equation}
Calculating explicitly the first and the last integral in~\eqref{eq:L_taylor_4}, we obtain
\[
e^{-LG(\gamma_2)} \frac{1-e^{L(G'(\gamma_2)+\varepsilon)\delta}}{-L(G'(\gamma_2)+\varepsilon)} 
\leq \int_{\gamma_2-\delta}^{\gamma_2} e^{-LG(\eta)}\ \mathrm{d}\eta 
\leq e^{-LG(\gamma_2)} \frac{1-e^{L(G'(\gamma_2)-\varepsilon)\delta}}{-L(G'(\gamma_2)-\varepsilon)}. 
\]
In this way, since $G'(\gamma_2)-\varepsilon>0$, we deduce that
\begin{equation}
    \label{eq:L_taylor_5}
    \frac{G'(\gamma_2)}{G'(\gamma_2)+\varepsilon} \leq \liminf_{L\to -\infty} \left(\frac{e^{-LG(\gamma_2)}}{-LG'(\gamma_2)} \right)^{-1} \int_{\gamma_2-\delta}^{\gamma_2} e^{-LG(\eta)}\ \mathrm{d}\eta \leq \limsup_{L\to -\infty} \left(\frac{e^{-LG(\gamma_2)}}{-LG'(\gamma_2)} \right)^{-1} \int_{\gamma_2-\delta}^{\gamma_2} e^{-LG(\eta)}\ \mathrm{d}\eta \leq \frac{G'(\gamma_2)}{G'(\gamma_2)-\varepsilon}.
\end{equation}
Joining~\eqref{eq:L_taylor_3} and~\eqref{eq:L_taylor_5}, we obtain that
\[
\frac{G'(\gamma_2)}{G'(\gamma_2)+\varepsilon} \leq \liminf_{L\to -\infty} \left(\frac{e^{-LG(\gamma_2)}}{-LG'(\gamma_2)} \right)^{-1} \int_{\gamma_1}^{\gamma_2} e^{-LG(\eta)}\ \mathrm{d}\eta \leq \limsup_{L\to -\infty} \left(\frac{e^{-LG(\gamma_2)}}{-LG'(\gamma_2)} \right)^{-1} \int_{\gamma_1}^{\gamma_2} e^{-LG(\eta)}\ \mathrm{d}\eta \leq \frac{G'(\gamma_2)}{G'(\gamma_2)-\varepsilon}.
\]
Our claim follows immediately by tending $\varepsilon$ to 0.
\end{proof}

Now, we introduce some notation. To make explicit the dependence on $L$, for any $L\in \R$, we define, as in~\eqref{eq:def_Psi}, the function
\begin{equation}
\label{eq:def_Psi_L}
     \Psi_L(s) \coloneqq \int_0^s e^{-\frac{1}{p-1} L G(\eta)} \ \mathrm{d}\eta, \ \forall s\in \R,
\end{equation}
where $G(s)= \int_0^s g(\eta)\ \mathrm{d}\eta$ for any $s\in \R$. We point out that $G$ is increasing when $g$ is positive. In relation to the area condition~\eqref{eq:area_cond}, for any $\gamma_1,\gamma_2 \in [\alpha, \beta]$ with $\gamma_1<\gamma_2$, we define the continuous function
\begin{equation}
    \label{eq:def_Hgamma}
    H_{\gamma_1,\gamma_2}(s,L) \coloneqq \int_{s}^{\gamma_2} f(\eta) e^{-\frac{p}{p-1}L G(\eta)} \ \mathrm d \eta, \ \forall (s,L) \in [0,\gamma_1]\times \R,
\end{equation}
and we set, for any $L\in\R$,
\begin{equation}
\label{eq:def_hmin_hmax}
    \uhgamma \coloneqq \min_{s\in [0,\gamma_1]} H_{\gamma_1,\gamma_2} (s,L) \qquad \text{and} \qquad \ohgamma \coloneqq \max_{s\in [0,\gamma_1]} H_{\gamma_1,\gamma_2} (s,L).
\end{equation}

The next result collects some asymptotic properties of these functions when $L\to -\infty$.

\begin{lemma}
\label{lem:technical}
    Let $f\in C(\R)$ satisfy $f(s)>0$ for $s\in (\alpha,\beta)$, where $0<\alpha<\beta$ are two zeros of $f$, and let $g\in C(\R)$ be positive. Let $\gamma_1,\gamma_2 \in [\alpha,\beta]$ be such that $\gamma_1< \gamma_2$, and let $\Psi_L(s)$, $\uhgamma$ and $\ohgamma$ be the functions defined in~\eqref{eq:def_Psi_L} and~\eqref{eq:def_hmin_hmax}. Then it holds:
    \begin{enumerate}[i)]
        \item If $\gamma_2<\beta$, there is some $L_0<0$ and some $C>0$ such that
        \[
        \uhgamma \geq C\frac{e^{-\frac{p}{p-1}LG(\gamma_2)}}{-L}, \ \forall L<L_0.
        \]
        If $\gamma_2=\beta$, then for any $\varepsilon>0$ small there is some $L_0=L_0(\varepsilon)<0$ and some $C=C(\varepsilon)>0$ such that
        \[
        \underline{h}_{\gamma_1,\beta} (L) \geq C\frac{e^{{-\frac{p}{p-1}LG(\beta-\varepsilon)}}}{-L}, \ \forall L<L_0.
        \]
        \item Both $\uhgamma$ and $\ohgamma$ grow at the same rate as $L\to -\infty$, i.e.,
        \[
        \lim_{L\to -\infty} \frac{\ohgamma}{\uhgamma} = 1.
        \]
        \item If $\gamma_2<\beta$, then
        \[
        \lim_{L\to -\infty} \frac{\Psi_L(\gamma_2)^p}{\uhgamma} = 0.
        \]
    \end{enumerate}
\end{lemma}
\begin{proof}
    During the whole proof, we assume $L<0$. Each item is proved as follows:
    \begin{enumerate}[i)]
    \item Assume first that $\gamma_2<\beta$. Let $\varepsilon_0\geq 0$ be such that $\gamma_1+\varepsilon_0\in (\alpha,\gamma_2)$. Taking into account that $\min_{t\in [0,\gamma_1]} f(t)\leq 0$ and that $\min_{t\in [\gamma_1+\varepsilon_0,\gamma_2]} f(t)> 0$, we deduce for any $s\in[0,\gamma_1]$ and any $L<0$ that
    \begin{equation}
        \label{eq:Pf_L_tec_1}
        \begin{split}
        H_{\gamma_1,\gamma_2}(s,L) =& \int_{s}^{\gamma_1} f(\eta) e^{-\frac{p}{p-1}L G(\eta)} \ \mathrm d \eta + \int_{\gamma_1}^{\gamma_2} f(\eta) e^{-\frac{p}{p-1}L G(\eta)} \ \mathrm d \eta\\
        \geq& (\gamma_1-s) e^{-\frac{p}{p-1}L G(\gamma_1)} \min_{t\in [0,\gamma_1]} f(t) +\int_{\gamma_1+\varepsilon_0}^{\gamma_2} f(\eta) e^{-\frac{p}{p-1}L G(\eta)} \ \mathrm d \eta \\
        \geq& \gamma_1 e^{-\frac{p}{p-1}L G(\gamma_1)} \min_{t\in [0,\gamma_1]} f(t) + \min_{t\in [\gamma_1+\varepsilon_0,\gamma_2]} f(t) \int_{\gamma_1+\varepsilon_0}^{\gamma_2} e^{-\frac{p}{p-1}L G(\eta)} \ \mathrm d \eta \\
        =& \frac{e^{-\frac{p}{p-1}LG(\gamma_2)}}{-L} \left( -\gamma_1 L e^{\frac{p}{p-1}L(G(\gamma_2) - G(\gamma_1))} \min_{t\in [0,\gamma_1]} f(t) 
        \vphantom{\left( \frac{e^{-\frac{p}{p-1}LG(\gamma)}}{-L} \right)^{-1}} \right.\\
        & + \left.\left( \frac{e^{-\frac{p}{p-1}LG(\gamma_2)}}{-L} \right)^{-1} \min_{t\in [\gamma_1+\varepsilon_0,\gamma_2]} f(t) \int_{\gamma_1+\varepsilon_0}^{\gamma_2} e^{-\frac{p}{p-1}L G(\eta)} \ \mathrm d \eta \right).
        \end{split}
    \end{equation}
    Therefore, for any $L<0$ we obtain that
    \begin{equation*}
    \begin{split}
        \uhgamma \geq& \frac{e^{-\frac{p}{p-1}LG(\gamma_2)}}{-L} \left( -\gamma_1 L e^{\frac{p}{p-1}L(G(\gamma_2) - G(\gamma_1))} \min_{t\in [0,\gamma_1]} f(t) 
        \vphantom{\left( \frac{e^{-\frac{p}{p-1}LG(\gamma)}}{-L} \right)^{-1}} \right.\\
        & + \left.\left( \frac{e^{-\frac{p}{p-1}LG(\gamma_2)}}{-L} \right)^{-1} \min_{t\in [\gamma_1+\varepsilon_0,\gamma_2]} f(t) \int_{\gamma_1+\varepsilon_0}^{\gamma_2} e^{-\frac{p}{p-1}L G(\eta)} \ \mathrm d \eta \right).
    \end{split}
    \end{equation*}
    Using Lemma~\ref{lem:taylor} and that $G(\gamma_1) < G(\gamma_2)$, we deduce our claim because
    \begin{equation*}
    \begin{split}
    \lim_{L\to-\infty} & \left( -\gamma_1 L e^{\frac{p}{p-1}L(G(\gamma_2) - G(\gamma_1))} \min_{t\in [0,\gamma_1]} f(t) + \left( \frac{e^{-\frac{p}{p-1}LG(\gamma_2)}}{-L} \right)^{-1} \min_{t\in [\gamma_1+\varepsilon_0,\gamma_2]} f(t) \int_{\gamma_1+\varepsilon_0}^{\gamma_2} e^{-\frac{p}{p-1}L G(\eta)} \ \mathrm d \eta \right)\\
    &= \frac{p-1}{p G'(\gamma_2)}\min_{t\in [\gamma_1+\varepsilon_0,\gamma_2]} f(t) >0.
    \end{split}
    \end{equation*}
    When $\gamma_2=\beta$, it suffices to observe that for any $\varepsilon\in (0,\beta-\gamma_1)$, we have
    \[
    \underline{h}_{\gamma_1,\beta}(L) \geq \underline{h}_{\gamma_1,\beta-\varepsilon}(L),\ \forall L\in \R,
    \]
    and then the result for the case $\gamma_2<\beta$ can be applied with $\gamma_2=\beta-\varepsilon$ to conclude our claim.

    \item To unify the proof for the cases $\gamma_2<\beta$ and $\gamma_2=\beta$ observe that, thanks to item i) and to the fact that $G$ is increasing, given $\varepsilon>0$ with $\gamma_2-\varepsilon\in (\gamma_1,\beta)$, there is $L_0<0$ and $C>0$ such that
    \begin{equation}
    \label{eq:Pf_L_tec_2}
        \uhgamma \geq C\frac{e^{-\frac{p}{p-1}LG(\gamma_2-\varepsilon)}}{-L}, \ \forall L<L_0.
    \end{equation}    
    We point out that, for this range of $L$, $\uhgamma$ is positive. For any $L\in \R$, we define $s_L\in [0,\gamma_1]$ as the point where $H_{\gamma_1,\gamma_2}(s,L)$ achieves its maximum in $s$. Then, it holds
    \begin{equation}
    \label{eq:Pf_L_tec_3}
        \ohgamma = \max_{s\in [0,\gamma_1]} H_{\gamma_1,\gamma_2} (s,L) = H_{\gamma_1, \gamma_2}(s_L,L) = \int_{s_L}^{\gamma_1} f(\eta) e^{-\frac{p}{p-1}L G(\eta)} \ \mathrm d \eta + \int_{\gamma_1}^{\gamma_2} f(\eta) e^{-\frac{p}{p-1}L G(\eta)} \ \mathrm d \eta.
    \end{equation}    
    On the one hand, using~\eqref{eq:Pf_L_tec_2} and that $\max_{t\in [0,\gamma_1]} f(t) \geq 0$, we have that
    \begin{equation*}
    \begin{split}
    0 &\leq \frac{\int_{s_L}^{\gamma_1} f(\eta) e^{-\frac{p}{p-1}L G(\eta)} \ \mathrm d \eta}{\uhgamma} 
    \leq \frac{\gamma_1 e^{-\frac{p}{p-1}L G(\gamma_1)} \max_{t\in [0,\gamma_1]} f(t)}{\uhgamma}\\
    &\leq \frac{\gamma_1 \max_{t\in [0,\gamma_1]} f(t)}{C} \frac{-L}{e^{-\frac{p}{p-1} L(G(\gamma_2-\varepsilon)-G(\gamma_1))}},\ \forall L<L_0.
    \end{split}
    \end{equation*}
    As $G(\gamma_2-\varepsilon)>G(\gamma_1)$, we immediately obtain that
    \begin{equation}
    \label{eq:Pf_L_tec_4}
        \lim_{L\to -\infty} \frac{\int_{s_L}^{\gamma_1} f(\eta) e^{-\frac{p}{p-1}L G(\eta)} \ \mathrm d \eta}{\uhgamma} = 0.
    \end{equation}
    On the other hand, reasoning as in~\eqref{eq:Pf_L_tec_1}, we deduce that
    \[
    \uhgamma \geq  \gamma_1 e^{-\frac{p}{p-1}L G(\gamma_1)} \min_{t\in [0,\gamma_1]} f(t) + \int_{\gamma_1}^{\gamma_2} f(\eta) e^{-\frac{p}{p-1}L G(\eta)} \ \mathrm d \eta >0, \ \forall L<L_0,
    \]
    and then
    \begin{equation}
    \label{eq:Pf_L_tec_5}
    \begin{split}
        1 &\leq \frac{\int_{\gamma_1}^{\gamma_2} f(\eta) e^{-\frac{p}{p-1}L G(\eta)} \ \mathrm d \eta}{\uhgamma} \\
        &\leq \frac{\int_{\gamma_1}^{\gamma_2} f(\eta) e^{-\frac{p}{p-1}L G(\eta)} \ \mathrm d \eta}{\gamma_1 e^{-\frac{p}{p-1}L G(\gamma_1)} \min_{t\in [0,\gamma_1]} f(t) + \int_{\gamma_1}^{\gamma_2} f(\eta) e^{-\frac{p}{p-1}L G(\eta)} \ \mathrm d \eta}\\[2mm]
        &= \frac{1}{\frac{\gamma_1 e^{-\frac{p}{p-1}L G(\gamma_1)} \min_{t\in [0,\gamma_1]} f(t)}{\int_{\gamma_1}^{\gamma_2} f(\eta) e^{-\frac{p}{p-1}L G(\eta)} \ \mathrm d \eta} + 1},\ \forall L<L_0.
    \end{split}
    \end{equation} 
    Now, given $\varepsilon_0>0$ such that $\gamma_1+\varepsilon_0\in (\alpha,\beta)$ and $\gamma_1+\varepsilon_0<\gamma_2-\varepsilon$, observe that, since $\min_{t\in [0,\gamma_1]} f(t)\leq 0$ and $\min_{t\in [\gamma_1+\varepsilon_0, \gamma_2-\varepsilon]} f(t)>0$, one has
    \begin{align*}
        0&\geq \frac{\gamma_1 e^{-\frac{p}{p-1}L G(\gamma_1)} \min_{t\in [0,\gamma_1]} f(t)}{\int_{\gamma_1}^{\gamma_2} f(\eta) e^{-\frac{p}{p-1}L G(\eta)} \ \mathrm d \eta}\\
        &\geq \frac{\gamma_1 e^{-\frac{p}{p-1}L G(\gamma_1)} \min_{t\in [0,\gamma_1]} f(t)}{(\gamma_2-\varepsilon-\gamma_1-\varepsilon_0)e^{-\frac{p}{p-1}LG(\gamma_1+\varepsilon_0)} \min_{t\in [\gamma_1+\varepsilon_0, \gamma_2-\varepsilon]} f(t)} \\
        &= \frac{\gamma_1 \min_{t\in [0,\gamma_1]} f(t)}{(\gamma_2-\varepsilon-\gamma_1-\varepsilon_0)\min_{t\in [\gamma_1+\varepsilon_0, \gamma_2-\varepsilon]} f(t)} \, e^{\frac{p}{p-1} L(G(\gamma_1 + \varepsilon_0) - G(\gamma_1))}, \ \forall L<L_0.
    \end{align*}
    Since $G(\gamma_1+\varepsilon_0)>G(\gamma_1)$, we deduce that
    \[
    \lim_{L\to -\infty} \frac{\gamma_1 e^{-\frac{p}{p-1}L G(\gamma_1)} \min_{t\in [0,\gamma_1]} f(t)}{\int_{\gamma_1}^{\gamma_2} f(\eta) e^{-\frac{p}{p-1}L G(\eta)} \ \mathrm d \eta} = 0.
    \]
    Therefore, we can take limits in~\eqref{eq:Pf_L_tec_5} to obtain that
    \begin{equation}
    \label{eq:Pf_L_tec_6}
        \lim_{L\to-\infty} \frac{\int_{\gamma_1}^{\gamma_2} f(\eta) e^{-\frac{p}{p-1}L G(\eta)} \ \mathrm d \eta}{\uhgamma} = 1.
    \end{equation}
    Joining~\eqref{eq:Pf_L_tec_3},~\eqref{eq:Pf_L_tec_4} and~\eqref{eq:Pf_L_tec_6} we deduce our claim.

    \item By Lemma~\ref{lem:taylor}, there is some $C_1>0$ and some $L_1<0$ such that
    \[
    \Psi_L(\gamma_2) \leq \frac{C_1(p-1)}{G'(\gamma_2)} \frac{e^{-\frac{1}{p-1}LG(\gamma_2)}}{-L} ,\ \forall L<L_1.
    \]
    In this way, using item i), we deduce the existence of some $L_2<0$ and some $C_2>0$ such that
    \[
    0\leq \frac{\Psi_L(\gamma_2)^p}{\uhgamma} \leq C_2 \left( \frac{C_1(p-1)}{G'(\gamma_2)} \right)^p  \frac{e^{-\frac{p}{p-1}LG(\gamma_2)}}{e^{-\frac{p}{p-1}LG(\gamma_2)}} \frac{-L}{(-L)^p} = C_2 \left( \frac{C_1(p-1)}{G'(\gamma_2)} \right)^p \frac{1}{(-L)^{p-1}},\ \forall L<L_2.
    \]
    Taking the limit as $L \to -\infty$, we obtain the desired result. \qedhere
    \end{enumerate}
\end{proof}

The next result establishes that, given any $\lambda>0$ and any $\gamma_1\in [\alpha,\beta)$, one can always find an unbounded range of $L$'s such that problem~\eqref{eq:PbL} has a nonnegative solution whose maximum lies between $\gamma_1$ and $\beta$.

\begin{proposition}
\label{prop:hess}
Let $f\in C(\R)$ satisfy $f(0)\geq 0$ and $f(s)>0$ for $s\in (\alpha,\beta)$, where $0<\alpha<\beta$ are two zeros of $f$, and let $g\in C(\R)$ be positive. Then, for each $\overline\lambda>0$ and each $\gamma_1\in [\alpha,\beta)$, there is some $\overline{L} \in \R$ such that for any $\lambda > \overline{\lambda}$ and any $L<\overline{L}$ problem~\eqref{eq:PbL} has a nonnegative solution $u\in \cuo$ with $\|u\|_\co \in (\gamma_1, \beta]$.
\end{proposition}

\begin{proof}
Since we are interested only in nonnegative solutions less than $\beta$, we redefine $f$ and $g$, without changing the notation, as
\begin{equation*}
f(s) \coloneq \left\{ \begin{alignedat}{1}
(1+s) f(0) \quad &\mbox{if} \;\; -1\leq s< 0, \\
f(s) \quad &\mbox{if} \;\; 0\leq s<\beta, \\
0 \quad &\mbox{otherwise},
\end{alignedat}
\right. 
\qquad \quad \text{and} \quad \qquad
g(s) \coloneq \left\{ \begin{alignedat}{1}
(1+s) g(0) \quad &\mbox{if} \;\; -1\leq s< 0, \\
g(s) \quad &\mbox{if} \;\; 0\leq s<\beta, \\
(\beta+1-s)g(\beta) \quad &\mbox{if} \;\; \beta \leq s<\beta+1, \\
0 \quad &\mbox{otherwise}.
\end{alignedat}
\right.
\end{equation*}
We observe that the boundedness of $G$ implies that the range of $\Psi_L$, defined in~\eqref{eq:def_Psi_L}, is $\R$, and consequently $\operatorname{Dom}\left(\Psi_L^{-1} \right)=\R$. We then define, as in~\eqref{eq:def_tilde_f}, the function
\begin{equation*}
    \tilde{f}_L(s) \coloneqq f\left(\Psi_L^{-1}(s) \right) e^{-LG\left(\Psi_L^{-1}(s)\right)},\ \forall s\in \R,
\end{equation*}
and we denote $\tilde F_L (s) \coloneq \int_0^{s} \tilde f_L(\eta) \ \mathrm{d}\eta$ for any $s\in \R$. 

Let $\overline\lambda>0$ and let $\gamma_1 \in [\alpha,\beta)$. Thanks to Proposition~\ref{prop:equiv}, it suffices to show the existence of some $\overline{L}\in\R$ such that, for any $\lambda > \overline{\lambda}$ and any $L<\overline{L}$, problem
\begin{equation}
\label{eq:PbSemiCV_L}
\left\{ \begin{alignedat}{2}
-\Delta_p v  &= \lambda \tilde f_L (v) \quad &&\mbox{in} \;\; \Omega, \\
v &= 0 \quad &&\mbox{on} \;\; \partial\Omega,
\end{alignedat}
\right.
\end{equation}
admits a nonnegative solution $v\in \cuo$ with $\|v\|_\co \in (\Psi_L(\gamma_1), \Psi_L(\beta)]$.

Consider the functional
\[
I_{\lambda,L}(v)= \frac{1}{p} \into |\nabla v|^p - \lambda \into \tilde F_L(v),\ \forall v\in \sob.
\]
Since $\tilde F_L$ is continuous and bounded, the functional $I_{\lambda,L}$ is coercive and weakly lower semicontinuous for every $\lambda>0$ and every $L\in\R$. Therefore, $I_{\lambda,L}$ attains a global minimum at some $v_{\lambda,L}\in \sob$. As $v_{\lambda,L}$ is a solution to~\eqref{eq:PbSemiCV_L}, since $\tilde f_L(s)\geq 0$ when $s\leq 0$ and $\tilde f_L(s)=0$ when $s\geq \Psi_L(\beta)$, it can be easily deduced that $0\leq v_{\lambda,L}\leq \Psi_L(\beta)$ $\ae$ in $\Omega$. Moreover, the regularity theory (\cite{Lieb}) implies that $v_{\lambda,L} \in \cuo$.

In the following, we show that there is some $\overline{L}\in \R$ such that this minimizer has $\co$-norm in $(\Psi_L(\gamma_1), \Psi_L(\beta)]$ for any $\lambda>\overline{\lambda}$ and any $L<\overline{L}$. 

Assume, by contradiction, that there is some sequence $L_n\in \R$ with $L_n\to -\infty$ and some $\lambda_n>\overline{\lambda}$ such that $v_n\coloneq v_{\lambda_n,L_n}$ verifies $\|v_n\|_\co \leq \Psi_{L_n}(\gamma_1)$ for all $n\in \N$. Let $\gamma_2\in (\gamma_1,\beta)$ and consider an open set $\Omega_0 \subset\subset \Omega$. We take a cutoff function $w_n\in C_0^1 (\overline\Omega)$ such that $w_n\equiv \Psi_{L_n}(\gamma_2)$ in $\Omega_0$, $0\leq w_n \leq \Psi_{L_n}(\gamma_2)$ in $\Omega \setminus \Omega_0$, and $\|\nabla w_n\|_{C(\overline{\Omega}, \R^N)} \leq c_0 \Psi_{L_n}(\gamma_2)$, where $c_0>0$ is a constant not depending on $n$. We have that
\begin{equation}
\label{eq:Pf_P_Hess_1}
\begin{split}
\int_{\Omega} \left(\tilde F_{L_n}(w_n) - \tilde F_{L_n}(v_n) \right) 
&=
\int_{\Omega_0} \tilde F_{L_n}(w_n) +
\int_{\Omega \setminus \Omega_0} \tilde F_{L_n}(w_n) - \int_{\Omega} \tilde F_{L_n} (v_n)
\\
&=
\int_{\Omega_0} \tilde F_{L_n}(\Psi_{L_n}(\gamma_2)) +
\int_{\Omega \setminus \Omega_0} \tilde F_{L_n}(w_n) - \int_{\Omega} \tilde F_{L_n} (v_n)
\\
&=
\int_{\Omega} \tilde F_{L_n}(\Psi_{L_n}(\gamma_2)) - \int_{\Omega \setminus \Omega_0} \left(\tilde F_{L_n}(\Psi_{L_n}(\gamma_2))  - \tilde F_{L_n}(w_n) \right) - \into \tilde F_{L_n}(v_n)
\\
&=
\int_{\Omega} \int_{v_n}^{\Psi_{L_n}(\gamma_2)} \tilde f_{L_n}(\eta)\ \mathrm{d}\eta - \int_{\Omega \setminus \Omega_0} \int_{w_n}^{\Psi_{L_n}(\gamma_2)} \tilde f_{L_n}(\eta)\ \mathrm{d}\eta
\\[1mm]
&\geq \underline{h}_{\gamma_1,\gamma_2}(L_n)|\Omega| - \overline{h}_{\gamma_1,\gamma_2}(L_n) |\Omega \setminus \Omega_0|,
\end{split}
\end{equation}
where $\uhgamma$ and $\ohgamma$ are defined in~\eqref{eq:def_hmin_hmax}. Here, we have taken into account (reasoning as in~\eqref{eq:Pf_Th_ex_2}) that
\begin{align*}
\int_{w_n}^{\Psi_{L_n}(\gamma_2)} \tilde f_{L_n}(\eta) \ \mathrm{d}\eta 
&\leq \max_{0\leq s \leq \Psi_{L_n}(\gamma_2)} \left\{\int_s^{\Psi_{L_n}(\gamma_2)} \tilde f_{L_n}(\eta) \ \mathrm{d}\eta \right\}
= \max_{0\leq s \leq \gamma_2} \left\{\int_s^{\gamma_2} f(\rho) e^{-\frac{p}{p-1}L_n G(\rho)} \ \mathrm{d}\rho \right\}\\
&= \max_{0\leq s \leq \gamma_1} \left\{\int_s^{\gamma_2} f(\rho) e^{-\frac{p}{p-1}L_n G(\rho)} \ \mathrm{d}\rho \right\} 
= \overline{h}_{\gamma_1,\gamma_2}(L_n)
\end{align*}
and that, since $v_n\leq \Psi_{L_n}(\gamma_1)$, it holds
\begin{align*}
\int_{v_n}^{\Psi_{L_n}(\gamma_2)} \tilde f_{L_n}(\eta) \ \mathrm d \eta 
= \int_{\Psi_{L_n}^{-1}(v_n)}^{\gamma_2} f\left(\rho \right)  e^{-\frac{p}{p-1}L_n G(\rho)} \ \mathrm{d}\rho 
\geq \min_{0\leq s \leq \gamma_1} \left\{\int_s^{\gamma_2} f(\rho) e^{-\frac{p}{p-1}L_n G(\rho)} \ \mathrm{d}\rho \right\} = \underline{h}_{\gamma_1,\gamma_2}(L_n).
\end{align*} 
Taking~\eqref{eq:Pf_P_Hess_1} into account and using that $\|\nabla w_n\|_{{C(\overline{\Omega}, \R^N)}} \leq c_0 \Psi_{L_n}(\gamma_2)$, we deduce
\begin{equation}
\label{eq:Pf_P_Hess_2}
\begin{split}
I_{\lambda_n,L_n}(w_n) - I_{\lambda_n,L_n}(v_n) &= \frac{1}{p} \into (|\nabla w_n|^p-|\nabla v_n|^p) - \lambda_n \int_{\Omega} (\tilde F_{L_n}(w_n) - \tilde F_{L_n}(v_n))
\\
&\leq \frac{1}{p} \into |\nabla w_n|^p - \lambda_n \int_{\Omega} (\tilde F_{L_n}(w_n) - \tilde F_{L_n} (v_n))
\\
&\leq \frac{1}{p} \into |\nabla w_n|^p - \lambda_n \left( \underline{h}_{\gamma_1,\gamma_2}(L_n) |\Omega| - \overline{h}_{\gamma_1,\gamma_2}(L_n) |\Omega \setminus \Omega_0| \right)
\\
&\leq \frac{1}{p} c_0^p \Psi_{L_n}(\gamma_2)^p |\Omega| - \lambda_n \left( \underline{h}_{\gamma_1,\gamma_2}(L_n)|\Omega| - \overline{h}_{\gamma_1,\gamma_2}(L_n) |\Omega\setminus \Omega_0| \right)\\
&= \underline{h}_{\gamma_1,\gamma_2}(L_n) \left[ \frac{c_0^p |\Omega|}{p}\frac{\Psi_{L_n}(\gamma_2)^p}{\underline{h}_{\gamma_1,\gamma_2}(L_n)} - \lambda_n \left( |\Omega| -  |\Omega\setminus \Omega_0| \frac{\overline{h}_{\gamma_1,\gamma_2}(L_n)}{\underline{h}_{\gamma_1,\gamma_2}(L_n)} \right) \right].
\end{split}
\end{equation}
Since $L_n\to -\infty$, as a consequence of Lemma~\ref{lem:technical}, we can take $n_0\in\N$ such that
\[
|\Omega| -  |\Omega\setminus \Omega_0| \frac{\overline{h}_{\gamma_1,\gamma_2}(L_n)}{\underline{h}_{\gamma_1,\gamma_2}(L_n)} >0,\ \forall n\geq n_0.
\]
As $\lambda_n>\overline{\lambda}$, from~\eqref{eq:Pf_P_Hess_2} we deduce for any $n\geq n_0$ that
\[
I_{\lambda_n,L_n}(w_n) - I_{\lambda_n,L_n}(v_n) \leq \underline{h}_{\gamma_1,\gamma_2}(L_n) \left[ \frac{c_0^p |\Omega|}{p}\frac{\Psi_{L_n}(\gamma_2)^p}{\underline{h}_{\gamma_1,\gamma_2}(L_n)} - \overline{\lambda} \left( |\Omega| -  |\Omega\setminus \Omega_0| \frac{\overline{h}_{\gamma_1,\gamma_2}(L_n)}{\underline{h}_{\gamma_1,\gamma_2}(L_n)} \right) \right].
\]
Finally, by Lemma~\ref{lem:technical}, we have that $\lim_{n\to \infty} \underline{h}_{\gamma_1,\gamma_2}(L_n) = \infty$ and that
\[
\lim_{n\to \infty} \left[ \frac{c_0^p |\Omega|}{p}\frac{\Psi_{L_n}(\gamma_2)^p}{\underline{h}_{\gamma_1,\gamma_2}(L_n)} - \overline{\lambda} \left( |\Omega| -  |\Omega\setminus \Omega_0| \frac{\overline{h}_{\gamma_1,\gamma_2}(L_n)}{\underline{h}_{\gamma_1,\gamma_2}(L_n)} \right) \right] = \overline{\lambda}(|\Omega\setminus \Omega_0|-|\Omega|)<0.
\]
Therefore, we can find some $n_1\in \N$ such that
\[
I_{\lambda_n,L_n}(w_n) < I_{\lambda_n,L_n}(v_n), \ \forall n\geq n_1,
\]
but this contradicts the fact that $v_n$ is a minimizer of $I_{\lambda_n,L_n}$. 

Then, we conclude that there is some $\overline{L}\in \R$ such that $\Psi_L(\gamma_1) < \|v_{\lambda,L}\|_\co \leq\Psi_L(\beta)$ for any $\lambda>\overline{\lambda}$ and any $L<\overline{L}$, as desired.
\end{proof}

\subsection{Stability}

In some situations, it is necessary to pass to the limit in problems such as~\eqref{eq:PbL}. Although it is standard, we include the following result for the sake of completeness.

\begin{lemma}
\label{lem:stab}
Let $f, g\in C(\R)$. Let $u_n\in \cuo$ be a solution to
\begin{equation}
\label{eq:PbStab}
\left\{ \begin{alignedat}{2}
-\Delta_p u_n + L_n g(u_n) |\nabla u_n| ^p &= \lambda_n f(u_n) \quad &&\mbox{in} \;\; \Omega, \\
u_n &= 0 \quad &&\mbox{on} \;\; \partial\Omega, 
\end{alignedat}
\right.
\end{equation}
where $\lambda_n$ and $L_n$ are sequences of real numbers. If $L_n\to L\in \R$, $\lambda_n \to \lambda \in \R$, and for some $M>0$ it holds $\|u_n\|_\co \leq M$ for all $n\in \N$, then, up to a subsequence, $u_n\to u$ in $C(\overline{\Omega})$, where $u$ is a solution to~\eqref{eq:PbL}.
\end{lemma}

\begin{proof}
This result is a straightforward consequence of the seminal paper~\cite{Lieb}. Indeed, problem~\eqref{eq:PbStab} satisfies, uniformly in $n$, all the structure conditions required in~\cite[Theorem~1]{Lieb}. Therefore, there exists some $C>0$ and some $\mu\in (0,1)$ such that
\[
\|u_n\|_{C^{1,\mu}(\overline{\Omega})} \leq C,\ \forall n\in \N.
\]
As a consequence, $u_n$ has a subsequence, still denoted by $u_n$, such that $u_n\to u$ in $C(\overline \Omega)$ and $\nabla u_n \to \nabla u$ in $C(\overline \Omega, \R^N)$ for some $u\in \cuo$. Then one can pass to the limit in the weak formulation of~\eqref{eq:PbStab} to conclude that $u$ is a solution to~\eqref{eq:PbL}.
\end{proof}

\subsection{Structure of the solution set as \texorpdfstring{$L$}{L} varies}
Here, we aim to prove Theorem~\ref{th:Lfrak} and Theorem~\ref{th:conv_beta}. For clarity, we will subdivide Theorem~\ref{th:Lfrak} into smaller results, from which it will be easily deduced. In contrast, Theorem~\ref{th:conv_beta} will follow directly from Proposition~\ref{prop:hess}.

We start by studying for which values of $L$ problem~\eqref{eq:PbL} has, for some $\lambda>0$, a nonnegative solution with maximum in $[\alpha,\beta]$. In other words, we study the specific shape of the set $\mathfrak{L}$ defined in~\eqref{eq:def_Lfrak}.

\begin{proposition}
\label{prop:rangeL}
Let $f\in C(\R)$ satisfy $f(0)\geq 0$ and $f(s)>0$ for $s\in (\alpha,\beta)$, where $0<\alpha<\beta$ are two zeros of $f$, and let $g\in C(\R)$ be positive. If $\mathfrak{L}$ denotes the set defined in~\eqref{eq:def_Lfrak}, then there exists $\tilde L\in (-\infty,\infty]$ such that
\[
\mathfrak{L} = (-\infty,\tilde L).
\]
Moreover, $\tilde L = \infty$ if $f\geq 0$, whereas $\tilde L < \infty$ if $f$ changes sign in $[0,\beta]$.
\end{proposition}

\begin{proof}
For problem~\eqref{eq:PbL}, the area condition~\eqref{eq:area_cond} becomes
\begin{equation*}
    \int_{s}^\beta f(\eta) e^{-\frac{p}{p-1}L G(\eta)} \ \mathrm d \eta>0,\ \forall s\in [0,\beta).
\end{equation*} 
Observe that this area condition is verified for some $L\in \R$ if and only if $\underline{h}_{\alpha,\beta}(L)>0$, where $\underline{h}_{\alpha,\beta}(L)$ is defined in~\eqref{eq:def_hmin_hmax}. Due to Theorem~\ref{th:existence}, we have the equivalence
\[
\mathfrak{L}= \left\{L\in \R: \underline{h}_{\alpha,\beta}(L)>0 \right\}.
\]
The easiest case arises when $f\geq0$. Here, it is straightforward to see that $\underline{h}_{\alpha,\beta}(L)>0$ for any $L\in \R$. We deduce that $\mathfrak{L}=\R$ and hence $\tilde L= \infty$.

When $f$ changes sign in $[0,\beta]$, we begin by noting that $\mathfrak{L}$ is not empty. Indeed, thanks to item i) of Lemma~\ref{lem:technical}, we have $\lim_{L\to -\infty} \underline{h}_{\alpha,\beta} (L) = \infty$, which implies $\mathfrak{L}\neq \emptyset$. Furthermore, $\mathfrak{L}$ is bounded above. To prove this, let $s_0\in (0,\alpha)$ and $\varepsilon_0>0$ be such that $f(s) <0$ for all $s\in [s_0,s_0+2\varepsilon_0]$, and choose $c_0>0$ such that $f(s)\leq -c_0<0$ for any $s\in[s_0,s_0+\varepsilon_0]$. Then, using that $G$ is increasing, we have for $L>0$ that
\begin{align*}
    \underline{h}_{\alpha,\beta}(L) &\leq H_{\alpha,\beta}(s_0,L) = \int_{s_0}^{s_0+\varepsilon_0} f(\eta) e^{-\frac{p}{p-1}L G(\eta)} \ \mathrm d \eta + \int_{s_0+\varepsilon_0}^\beta f(\eta) e^{-\frac{p}{p-1}L G(\eta)} \ \mathrm d \eta\\
    &\leq -c_0 \varepsilon_0 e^{-\frac{p}{p-1}L G(s_0+\varepsilon_0)} + \int_{s_0+ 2\varepsilon_0}^\beta f(\eta) e^{-\frac{p}{p-1}L G(\eta)} \ \mathrm d \eta \\
    & \leq -c_0 \varepsilon_0 e^{-\frac{p}{p-1}L G(s_0+\varepsilon_0)} + (\beta - s_0-2\varepsilon_0) e^{-\frac{p}{p-1}L G(s_0+2\varepsilon_0)} \max_{t\in [s_0+2\varepsilon_0,\beta]} f(t) \\
    &= e^{-\frac{p}{p-1}L G(s_0+\varepsilon_0)} \left( -c_0 \varepsilon_0 + (\beta - s_0-2 \varepsilon_0) e^{-\frac{p}{p-1}L(G(s_0+2\varepsilon_0)-G(s_0+\varepsilon_0))} \max_{t\in [s_0+2\varepsilon_0,\beta]} f(t) \right).
\end{align*}
Since $G(s_0+2\varepsilon_0) > G(s_0+\varepsilon_0)$, then
\[
\lim_{L\to\infty} \left( -c_0 \varepsilon_0 + (\beta - s_0-2\varepsilon_0) e^{-\frac{p}{p-1}L(G(s_0+2\varepsilon_0)-G(s_0+\varepsilon_0))} \max_{t\in [s_0+2\varepsilon_0,\beta]} f(t) \right) = -c_0 \varepsilon_0.
\]
Hence, we deduce the existence of some $L_0>0$ such that $\underline{h}_{\alpha,\beta}(L)<0$ for any $L>L_0$. 

As a consequence, if we define
\[
\tilde L \coloneqq \sup \mathfrak{L}= \sup \left\{L\in \R: \underline{h}_{\alpha,\beta}(L) >0 \right\},
\]
then we have that $\tilde L\in \R$. 

It remains for us to show that $\mathfrak{L}=(-\infty, \tilde L)$. First, observe that $\mathfrak{L}$ is an open set because $\underline{h}_{\alpha,\beta}(L)$ is continuous. In particular, this implies that $\mathfrak{L}\subseteq (-\infty,\tilde L)$. Now, we prove that the reverse inclusion also holds. Let $L\in (-\infty, \tilde L)$. By definition of supremum, we can find some $L_1\in \mathfrak{L}$ such that $L<L_1$. As $L_1\in \mathfrak{L}$, this means that for some large $\lambda$ there is $0\leq u_1\in \cuo$ with $\|u_1\|_\co\in [\alpha, \beta]$ weak solution of
\begin{equation*}
\left\{ \begin{alignedat}{2}
-\Delta_p u + L_1 g(u) |\nabla u|^p &= \lambda f(u) \quad &&\mbox{in} \;\; \Omega, \\
u &= 0 \quad &&\mbox{on} \;\; \partial\Omega.
\end{alignedat}
\right.
\end{equation*}
Since $L<L_1$, we have that $u_1$ is a strict subsolution of~\eqref{eq:PbL}. As $\overline{u}\equiv \beta$ is a supersolution of~\eqref{eq:PbL}, by~\cite[Theorem~3.1]{BMP} there is a solution $u\in \cuo$ of~\eqref{eq:PbL} such that $u_1\leq u \leq \beta$ in $\Omega$. Then, $\|u\|_\co\in [\alpha, \beta]$. Due to the definition of $\mathfrak{L}$, we obtain that $L\in \mathfrak{L}$ and thus $(-\infty,\tilde L) \subseteq \mathfrak{L}$. Therefore, we conclude that $\mathfrak{L}=(-\infty,\tilde L)$.
\end{proof}

Under the assumptions of Theorem~\ref{th:Lfrak}, the sets
\begin{align*}
    \Lambda_L &\coloneq \left\{\lambda\geq 0: \eqref{eq:PbL} \text{ has a solution } 0\leq u\in \cuo \text{ with } \|u\|_\co\in [\alpha,\beta] \right\} \quad \text{and} \\
    \overline{\Lambda}_L &\coloneq \left\{\overline{\lambda}\geq 0: \eqref{eq:PbL} \text{ has a solution } 0\leq u\in \cuo \text{ with } \|u\|_\co\in [\alpha,\beta] \text{ for any } \lambda > \overline{\lambda} \right\}
\end{align*}
are nonempty and closed for any $L\in (-\infty,\tilde L)$ as a consequence of Lemma~\ref{lem:stab} and Proposition~\ref{prop:rangeL}. Then, for any $L\in (-\infty,\tilde L)$ one can define 
\[
\lambda_{\rm min}(L) \coloneqq \min \Lambda_L \quad \text{and} \quad \overline\lambda_{\rm min}(L) \coloneqq \min \overline\Lambda_L.
\]
By definition, it holds $\lambda_{\rm min}(L) \leq \overline\lambda_{\rm min}(L)$ for any $L\in (-\infty, \tilde L)$. Moreover, both quantities are always positive because for $\lambda=0$ the only solution to~\eqref{eq:PbSemiCV_L} is $v\equiv 0$.

In the following, our aim is to study the behaviour of $\lambda_{\rm min}(L)$ as $L\to \tilde L$. We distinguish two cases: when $f$ changes sign and $\tilde L<\infty$, and when $f\geq 0$ and $\tilde L = \infty$. For the latter, we employ a strategy introduced in~\cite{Arc-Rez-Sil}, which is based on the Stampacchia regularity method (\cite{Stamp}). The next lemma, whose proof can be found in~\cite[Lemma~4.1]{Stamp}, will be needed.

\begin{lemma}[\cite{Stamp}]
    \label{lem:Stamp}
    Let $k_0\geq 0$ and let $\Theta : [k_0, \infty) \to [0, \infty)$ be a nonincreasing function. If there exist $c, a > 0$ and $b > 1$ such that
    \[
    \Theta(h) \leq \frac{c}{(h - k)^a} [\Theta(k)]^b, \quad \forall h > k > k_0,
    \]
    then
    \[
    \Theta(k_0 + d) = 0, \quad \text{for } d^a := c 2^{\frac{ab}{b - 1}} [\Theta(k_0)]^{b - 1}.
    \]
\end{lemma}

Now, we are ready to study the behaviour of $\lambda_{\rm min}(L)$ as $L\to \tilde L$. 

\begin{proposition}
\label{prop:infty}
Let $f\in C(\R)$ satisfy $f(0)\geq 0$ and $f(s)>0$ for $s\in (\alpha,\beta)$, where $0<\alpha<\beta$ are two zeros of $f$, and let $g\in C(\R)$ be positive. Then, $\lambda_{\rm min}(L)\to \infty$ as $L\to \tilde L$.
\end{proposition}

\begin{proof}
Arguing by contradiction, suppose that $\lambda_{\rm min}(L) \not\to \infty$ as $L\to \tilde L$. Then there is some $L_n\to \tilde L$ such that $\lambda_n \coloneq \lambda_{\rm min}(L_n)$ is bounded, say by $\overline{\lambda}$. Since $\lambda_n\in \Lambda_{L_n}$, for each $n\in \N$ there exists a nonnegative solution $u_n\in \cuo$ to~\eqref{eq:PbStab} such that $\|u_n\|_\co\in [\alpha, \beta]$.

When $f$ changes sign, one has that $\tilde L < \infty$ (see Proposition~\ref{prop:rangeL}). Passing to a subsequence if necessary, we may assume that $\lambda_n\to \tilde \lambda$ for some $\tilde \lambda\geq 0$. By Lemma~\ref{lem:stab}, one can pass to the limit in~\eqref{eq:PbStab} to deduce that $\tilde L\in \mathfrak{L}$, but this is a contradiction with Proposition~\ref{prop:rangeL}.

Now, we study the case where $f\geq 0$. In this setting, one has $\tilde L = \infty$. First, we prove that $u_n$ converges to 0 in $\sob$. Given $\varepsilon>0$, we define the truncation function $T_\varepsilon(s) = \min\{s,\varepsilon\}$ for any $s\in \R$.
Taking $\frac{1}{\varepsilon} T_\varepsilon(u_n)$ as test function in~\eqref{eq:PbStab}, we deduce that
\[
\into \frac{1}{\varepsilon} |\nabla T_\varepsilon(u_n)|^p + L_n \into \frac{1}{\varepsilon}T_\varepsilon(u_n) g(u_n) |\nabla u_n|^p = \into \lambda_n f(u_n) \frac{1}{\varepsilon}T_\varepsilon(u_n).
\]
After dropping the first integral, we obtain
\[
L_n \min_{t\in[0,\beta]}g(t) \into \frac{1}{\varepsilon}T_\varepsilon(u_n) |\nabla u_n|^p \leq L_n \into \frac{1}{\varepsilon}T_\varepsilon(u_n) g(u_n) |\nabla u_n|^p \leq \into \lambda_n f(u_n) \frac{1}{\varepsilon}T_\varepsilon(u_n) \leq \overline{\lambda} |\Omega| \max_{t\in [0,\beta]}f(t).
\]
Note that $\min_{t\in[0,\beta]}g(t)>0$. If $C_1\coloneq \overline{\lambda} |\Omega| \max_{t\in [0,\beta]}f(t) \left(\min_{t\in[0,\beta]}g(t)\right)^{-1}>0$, one deduces tending $\varepsilon$ to 0 that
\[
L_n \into |\nabla u_n|^p = L_n \int_{\{u_n>0\}} |\nabla u_n|^p \leq C_1.
\]
Since $L_n\to\infty$, we can ensure that $u_n\to 0$ strongly in $\sob$.

In the following, we show that $u_n$ converges to 0 in $\co$. This will lead to a contradiction because $\|u_n\|_\co\in [\alpha,\beta]$ for every $n\in\N$. To this end, we use a strategy of~\cite{Arc-Rez-Sil} based on the Stampacchia regularity method. Since we make use of the Sobolev embeddings, we present here the proof only for the case $N>p$; the case $N=p$ can be treated similarly, while for $N<p$ the result follows directly from Morrey's Theorem.

Given $k>0$, we define the function $G_k(s)= \max\{ 0, s-k\}$ for all $s\in \R$. Taking $G_k(u_n)$ as test function in~\eqref{eq:PbStab}, we obtain that
\begin{equation}
\label{eq:Pf_P_Arcoya_1}
\into |\nabla G_k(u_n)|^p + L_n \into G_k(u_n) g(u_n) |\nabla u_n|^p = \into \lambda_n f(u_n) G_k(u_n).
\end{equation}
We set $\Omega_n(k)= \{x\in\Omega: u_n>k\}$. We point out that $\lambda_n f(u_n)$ is bounded in $\lio$ by $C_2\coloneq \overline{\lambda} \max_{t\in [0,\beta]} f(t)$. Dropping a nonnegative term in~\eqref{eq:Pf_P_Arcoya_1}, and using the H\"older inequality with exponents $p^* = \frac{pN}{N-p}$ and $(p^*)' = \frac{pN}{N(p-1)+p}$, we deduce that
\begin{equation*}
\begin{split}
\|G_k(u_n)\|_{\sob}^p &= \into |\nabla G_k(u_n)|^p \leq \into \lambda_n f(u_n) G_k(u_n) \leq C_2 \int_{\Omega_n(k)} G_k(u_n) \leq C_2 |\Omega_n(k)|^\frac{1}{(p^*)'} \|G_k(u_n)\|_{L^{p^*}(\Omega)}.
\end{split}
\end{equation*}
Using the Sobolev embedding $\sob \hookrightarrow L^{p^*}(\Omega)$, we find $C_3>0$ such that
\begin{equation}
\label{eq:Pf_P_Arcoya_2}
\|G_k(u_n)\|_{L^{p^*}(\Omega)}^p \leq C_3 |\Omega_n(k)|^\frac{1}{(p^*)'} \|G_k(u_n)\|_{L^{p^*}(\Omega)}.
\end{equation}
Observe that, if we take $h>k>0$, then $\Omega_n(h)\subseteq \Omega_n(k)$ for any $n\in \N$. Using that $G_k(s)\geq h-k$ for all $s\geq h$, we deduce that
\begin{equation}
\label{eq:Pf_P_Arcoya_3}
\|G_k(u_n)\|_{L^{p^*}(\Omega)} = \left(\int_{\Omega_n(k)} G_k(u_n)^{p^*} \right)^\frac{1}{p^*} \geq (h-k) |\Omega_n(h)|^\frac{1}{p^*}.
\end{equation}
Combining~\eqref{eq:Pf_P_Arcoya_2} and~\eqref{eq:Pf_P_Arcoya_3}, we obtain, for some $C_4>0$, that
\begin{equation}
\label{eq:Pf_P_Arcoya_4}
|\Omega_n(h)| \leq \frac{C_4}{(h-k)^{p^*}} |\Omega_n(k)|^{\frac{p^*}{(p-1) (p^*)'}}.
\end{equation}

Let $\varepsilon>0$ be fixed. We consider the distribution function of $u_n$ defined as $\Theta_n(k) = |\Omega_n(k)|$ for every $k\geq 0$. As $\Theta_n$ is nonincreasing, thanks to relation~\eqref{eq:Pf_P_Arcoya_4} we can apply Stampacchia's Lemma~\ref{lem:Stamp} with $k_0=\varepsilon$, $a=p^*$, $c=C_4$ and $b=\frac{p^*}{(p-1)(p^*)'} = \frac{N(p-1)+p}{(p-1)(N-p)}>1$ to obtain, for every $n\in \N$, that
\begin{equation}
\label{eq:Pf_P_Arcoya_5}
\Theta_n(\varepsilon+d_n) = |\Omega_n(\varepsilon+d_n)| = 0, \text{ with } d_n^a = c2^\frac{ab}{b-1} |\Omega_n(\varepsilon)|^{b-1}.
\end{equation}
Notice that, by the definition of $\Omega_n(k)$, for any $n\in\N$ we have that
\[
\varepsilon |\Omega_n(\varepsilon)| \leq \int_{\Omega_n(\varepsilon)} u_n \leq \|u_n\|_{\luo}.
\]
Since $u_n$ converges to 0 strongly in $\luo$, we can find $n_0\in \N$ (depending on $\varepsilon$) such that $d_n<\varepsilon$ for any $n\geq n_0$. As $\Theta_n$ is nonincreasing, from~\eqref{eq:Pf_P_Arcoya_5} we deduce that $|\Omega_n(2\varepsilon)|=0$ for all $n\geq n_0$ or, equivalently, that 
\[
\|u_n\|_\co \leq 2\varepsilon, \ \forall n\geq n_0.
\]
Hence, we conclude that $u_n\to 0$ strongly in $\co$, which contradicts the fact that $\|u_n\|_\co \in [\alpha,\beta]$ for every $n\in\N$.
\end{proof}

All the results established in this section reduce Theorem~\ref{th:Lfrak} to a mere corollary. Its proof is given below.

\begin{proof}[Proof of Theorem~\ref{th:Lfrak}]
Thanks to Proposition~\ref{prop:rangeL} and Proposition~\ref{prop:infty}, it only remains for us to show that $\overline{\lambda}_{\rm min}(L)\to 0$ as $L\to -\infty$. Given any $\overline{\lambda}>0$, by Proposition~\ref{prop:hess}, we can always find some $\overline{L}\in\R$ such that, for any $L<\overline{L}$ and any $\lambda>\overline{\lambda}$, a nonnegative solution $u\in \cuo$ to~\eqref{eq:PbL} with $\|u\|_\co \in (\alpha,\beta]$ exists. Then, by the definition of $\overline{\lambda}_{\rm min}(L)$, we have that
\[
\overline{\lambda}_{\rm min}(L) \leq \overline{\lambda}, \ \forall L<\overline{L}.
\]
Thus, we obtain that $\overline{\lambda}_{\rm min}(L)\to 0$ as $L\to -\infty$.
\end{proof}

To conclude, we fix $\lambda>0$ and we show that maximal solutions to~\eqref{eq:PbL} in $[0,\beta]$ exist when $L$ is less than a constant. Furthermore, we analyse their behaviour as $L\to -\infty$.

\begin{proof}[Proof of Theorem~\ref{th:conv_beta}]
Let $\lambda>0$ fixed. As a consequence of Proposition~\ref{prop:hess}, we can find some $L_\lambda \in \R$ such that for every $L<L_\lambda$ problem~\eqref{eq:PbL} has a nonnegative solution $u_L\in \cuo$ with $\co$-norm in $(\alpha,\beta]$. Since $\beta$ is a supersolution, by~\cite[Theorem~4.2]{BMP} there is a maximal solution $\overline u_L\in \cuo$ to~\eqref{eq:PbL} in the interval $[0,\beta]$ such that $u_L\leq \overline{u}_L \leq \beta$.

Finally, thanks to Proposition~\ref{prop:hess}, for any $\varepsilon>0$ we can find some $\overline{L}< L_\lambda$ such that $\|\overline u_L\|_\co \in (\beta-\varepsilon, \beta]$ for any $L<\overline{L}$. Then, we conclude that $\|\overline{u}_L\|_{\co} \to \beta \text{ as } L\to -\infty$.
\end{proof}

\begin{remark}
All the results in this section also hold when $Lg(u)$ in problem~\eqref{eq:PbL_Aux} is replaced by $Lg_1(u)+g_2(u)$, where $g_1\in C(\R)$ is positive and $g_2\in C(\R)$ may change sign.
\end{remark}

\section{Generalizations and further results}
\label{sec:further}

\subsection{A general divergence problem}
\label{sec:equivalent}

All the results of this work also hold for the more general problem
\begin{equation}
\label{eq:PbExtension}
\left\{ \begin{alignedat}{2}
-\operatorname{div} (a(u) |\nabla u|^{p-2} \nabla u) + g(u) |\nabla u|^p &= \lambda f(u) \quad &&\mbox{in} \;\; \Omega, \\
u &= 0 \quad &&\mbox{on} \;\; \partial\Omega.
\end{alignedat}
\right.
\end{equation}
Here, $f\in C(\R)$ satisfies $f(s)>0$ for all $s\in (\alpha,\beta)$, where $0<\alpha<\beta$ are two zeros of $f$; $g\in C(\R)$; and $a\in C(\R)$ is such that $a(s)>0$ for all $s\in \R$. Observe that for $a\equiv 1$ one recovers the usual $p$-Laplacian operator. Although the proofs are similar, problem~\eqref{eq:PbExtension} presents some differences with respect to~\eqref{eq:PbCV}.  For~\eqref{eq:PbExtension}, the function $\Psi$ used for the change of variables (cf.~\eqref{eq:def_Psi}) is
\[
\Psi(s) \coloneqq \int_0^s a(\eta)^\frac{1}{p-1} e^{-\frac{1}{p-1}\int_0^\eta \frac{g(\sigma)}{a(\sigma)} \, \mathrm{d}\sigma} \ \mathrm{d}\eta,\ \forall s\in \R.
\]
The function $\tilde f$ defined in~\eqref{eq:def_tilde_f} also changes. Indeed, Proposition~\ref{prop:equiv} still holds for~\eqref{eq:PbExtension}, but now, if $u$ is a solution to~\eqref{eq:PbExtension}, then $v=\Psi(u)$ solves~\eqref{eq:PbSemiCV} with $\tilde f$ given by
\[
\tilde{f}(s) \coloneqq f\left(\Psi^{-1}(s) \right) e^{-\int_0^{\Psi^{-1}(s)} \frac{g(\sigma)}{a(\sigma)} \, \mathrm{d}\sigma},\ \forall s\in \operatorname{Dom}\big(\Psi^{-1} \big).
\]
As these two functions change, so does the area condition. Arguing as in Theorem~\ref{th:existence}, one can show that the area condition for~\eqref{eq:PbExtension} is
\begin{equation}
\label{eq:ac_extension_1}
    \int_{s}^\beta f(\eta) a(\eta)^\frac{1}{p-1} e^{-\frac{p}{p-1}\int_0^\eta \frac{g(\sigma)}{a(\sigma)} \, \mathrm d\sigma} \ \mathrm d \eta > 0,\ \forall s\in [0,\beta).
\end{equation}
Note that when $a\equiv 1$ the condition obtained coincides with~\eqref{eq:area_cond}. The following result generalizes Theorem~\ref{th:existence}.

\begin{theorem}
Let $f\in C(\R)$ satisfy $f(s)>0$ for $s\in (\alpha,\beta)$, where $0<\alpha<\beta$ are two zeros of $f$, and let $a, g\in C(\R)$ with $a$ positive. Then the following holds:
\begin{enumerate}[i)]
    \item If $f$ verifies~\eqref{eq:ac_extension_1} and $f(0)\geq 0$, then there is some $\overline{\lambda}>0$ such that, for every $\lambda> \overline{\lambda}$, problem~\eqref{eq:PbExtension} has a nonnegative solution $u\in\cuo$ with $\|u\|_\co \in (\alpha,\beta]$. 
    \item If $f$ does not satisfy~\eqref{eq:ac_extension_1}, then problem~\eqref{eq:PbExtension} admits no nonnegative solution with maximum in $[\alpha,\beta]$ for any $\lambda>0$.
\end{enumerate}
\end{theorem}

In the same way, Theorem~\ref{th:Lfrak} and Theorem~\ref{th:conv_beta} also hold for~\eqref{eq:PbExtension} if one includes a parameter $L$ multiplying the gradient term. One of the keys is that, since $g/a$ is positive (because $g$ is assumed to be positive), the function $\int_0^s \frac{g(\eta)}{a(\eta)} \ \mathrm{d}\eta$ is increasing for any $s\in \R$. This allows us to replicate all the arguments of Section~\ref{sec:regularizing}.

An important particular case of problem~\eqref{eq:PbExtension} is when $g=\frac{1}{p}a'$. In this case, problem~\eqref{eq:PbExtension} has a variational structure and its associated Euler-Lagrange functional is
\[
J(u) = \frac{1}{p}\into a(u) |\nabla u|^p - \into F(u),\ \forall u\in\sob,
\]
where $F(s) = \int_0^s f(\eta)\ \mathrm{d}\eta$ for any $s\in \R$. This functional is well defined because, since we are only interested in solutions contained in $[0,\beta]$, one can assume without loss of generality that $a$ is bounded. Here, the area condition obtained by substituting in~\eqref{eq:ac_extension_1} is
\begin{equation}
\label{eq:ac_extension_2}
\int_{s}^\beta f(\eta) \ \mathrm d \eta > 0,\ \forall s\in [0,\beta).
\end{equation}
Therefore, in this case the usual area condition is recovered. In fact, assuming~\eqref{eq:ac_extension_2}, the variational arguments of~\cite[Lemma~2.1]{Cl-Sw} can be replicated to show the first part of Theorem~\ref{th:existence} without making a change of variables.

In this family of variational problems, we also find the celebrated Schr\"odinger operator (\cite{Schrod-1}). In its most basic form, it takes the shape $Lu = -\Delta u - u \Delta (u^2)$. For the $p$-Laplacian, one can consider its extension
\begin{equation*}
\left\{ \begin{alignedat}{2}
-\Delta_p u - \frac{\kappa}{2} |u|^{\kappa-2} u \Delta_p (|u|^\kappa) &= \lambda f(u) \quad &&\mbox{in} \;\; \Omega, \\
u &= 0 \quad &&\mbox{on} \;\; \partial\Omega,
\end{alignedat}
\right.
\end{equation*}
where $\kappa\geq 1+\frac{1}{p}$. This problem can be rewritten as~\eqref{eq:PbExtension} with $a(s) = 1 + \frac{\kappa^p}{2} |s|^{p(\kappa-1)}$ and $g=\frac{1}{p}a'$. Its associated functional is
\[
J(u) = \frac{1}{p}\into \left( 1 + \frac{\kappa^p}{2} |u|^{p(\kappa-1)} \right) |\nabla u|^p - \into F(u),\ \forall u\in\sob,
\]
and the area condition is~\eqref{eq:ac_extension_2}. Finally, we point out that our work generalizes some of the results of~\cite{dS-Sil}, where the authors establish Theorem~\ref{th:existence} for the classical Schr\"odinger operator $Lu = -\Delta u - u \Delta (u^2)$.

\subsection{More general first order terms}

Our results can be applied to establish the existence of solutions for non-autonomous problems. Consider the problem
\begin{equation}
\label{eq:PbNoCV}
\left\{ \begin{alignedat}{2}
-\Delta_p u + \tilde g(x,u) |\nabla u|^p &= \lambda f(u) \quad &&\mbox{in} \;\; \Omega, \\
u &= 0 \quad &&\mbox{on} \;\; \partial\Omega.
\end{alignedat}
\right.
\end{equation}
We stress that, for this problem, no change of variables is available. Assume that $\tilde g\colon \Omega\times \R \to \R$ is a Carath\'eodory function (that is, measurable with respect to $x$ for every $s\in \R$, and continuous with respect to $s$ for almost every $x\in\Omega$) such that
\begin{equation}
\label{eq:hyp_tilde_g_1}
-M \leq \tilde g(x,s) \leq g(s),\ \forall (x,s)\in \Omega \times \R,
\end{equation}
where $M>0$ and $g\in C(\R)$.

If $u$ is a nonnegative solution to~\eqref{eq:PbCV} with $\|u\|_\co \in [\alpha,\beta]$, then $u$ is a subsolution to~\eqref{eq:PbNoCV}. A sub-supersolution argument (\cite{BMP}) with $\beta$ as a supersolution yields the existence of a solution $\tilde u$ to~\eqref{eq:PbNoCV} such that $u\leq \tilde u \leq \beta$. Hence, as a consequence of Theorem~\ref{th:existence}, we can state the following result.

\begin{theorem}
Let $f\in C(\R)$ satisfy $f(s)>0$ for $s\in (\alpha,\beta)$, where $0<\alpha<\beta$ are two zeros of $f$, and let $\tilde g(x,s)$ be a Carath\'eodory function satisfying~\eqref{eq:hyp_tilde_g_1}. If $f(s)$ verifies~\eqref{eq:area_cond} and $f(0)\geq 0$, then there is some $\overline{\lambda}>0$ such that, for every $\lambda> \overline{\lambda}$, problem~\eqref{eq:PbNoCV} has a nonnegative solution $\tilde u\in\cuo$ with $\|\tilde u\|_\co \in (\alpha,\beta]$. 
\end{theorem}

The regularizing effect is also present in this case. Let us introduce, as in~\eqref{eq:PbL}, a parameter $L\in \R$ multiplying the gradient term, so that we have
\begin{equation}
\label{eq:PbNoCV_L}
\left\{ \begin{alignedat}{2}
-\Delta_p u + L\tilde g(x,u) |\nabla u|^p &= \lambda f(u) \quad &&\mbox{in} \;\; \Omega, \\
u &= 0 \quad &&\mbox{on} \;\; \partial\Omega.
\end{alignedat}
\right.
\end{equation}
In this case, we assume the existence of $M_1,M_2>0$ such that
\begin{equation}
\label{eq:hyp_tilde_g_2}
M_1 \leq \tilde g(x,s) \leq M_2,\ \forall (x,s)\in \Omega \times \R.
\end{equation}
Observe that the maximal solutions to~\eqref{eq:PbCV} with $g\equiv M_1$ given by Theorem~\ref{th:conv_beta} are subsolutions to~\eqref{eq:PbNoCV_L} provided $L<0$. Using again the sub-supersolution method (\cite{BMP}) with $\beta$ as supersolution, one obtains the following.

\begin{theorem}
Let $f\in C(\R)$ satisfy $f(0)\geq 0$ and $f(s)>0$ for $s\in (\alpha,\beta)$, where $0<\alpha<\beta$ are two zeros of $f$, and let $\tilde g(x,s)$ be a Carath\'eodory function satisfying~\eqref{eq:hyp_tilde_g_2}. Given $\lambda>0$, there is some $L_\lambda<0$ such that the maximal solution $\overline u_L$ in the interval $[0,\beta]$ to problem~\eqref{eq:PbNoCV_L} exists for every $L\leq L_\lambda$, and
\[
\|\overline{u}_L\|_{\co} \to \beta \text{ as } L\to -\infty.
\]
\end{theorem}

\subsection{Nonexistence of solutions with maximum equal to \texorpdfstring{$\beta$}{beta}}
\label{sec:smp}

In Theorem~\ref{th:existence}, we establish the existence of nonnegative solutions $u\in\cuo$ to~\eqref{eq:PbCV} whose maximum belongs to $(\alpha,\beta]$. Our aim here is to discuss conditions under which $\|u\|_\co \neq \beta$.

As shown in Proposition~\ref{prop:equiv}, problem~\eqref{eq:PbCV} is closely related to the problem without the gradient term~\eqref{eq:PbSemi}. If $f$ satisfies
\begin{equation}
    \label{eq:hyp_smp}
    f(s) \leq M(\beta-s)^{p-1},\ \forall s\in [0,\beta],
\end{equation}
for some $M>0$, or equivalently
\begin{equation}
    \label{eq:hyp_smp_2}
    \limsup_{s\to \beta^-} \frac{f(s)}{(\beta-s)^{p-1}} <\infty,
\end{equation}
then the strong maximum principle precludes the existence of any solution $u$ to~\eqref{eq:PbSemi} with $\|u\|_\co =\beta$. Although we believe this result is well known, we include here its proof for the sake of completeness.

\begin{lemma}
\label{lem:smp}
    Let $f\in C(\R)$ satisfy~\eqref{eq:hyp_smp} and let $\lambda>0$. Then there is no nonnegative solution $u\in \cuo$ to~\eqref{eq:PbSemi} with $\|u\|_\co = \beta$.
\end{lemma}

\begin{proof}
Suppose that there exists a nonnegative solution $u\in\cuo$ to~\eqref{eq:PbSemi} with $\|u\|_\co = \beta$. Define $v\coloneq \beta-u$. Using~\eqref{eq:hyp_smp}, we obtain that $v\geq 0$ satisfies the problem
\begin{equation*}
\left\{ \begin{alignedat}{2}
-\Delta_p v + \lambda M v^{p-1}  &\geq 0 \quad &&\mbox{in} \;\; \Omega, \\
v &\geq 0 \quad &&\mbox{on} \;\; \partial\Omega.
\end{alignedat}
\right.
\end{equation*}
Since $v\not\equiv 0$ (because $v=\beta$ on $\partial\Omega$), the strong maximum principle (\cite{Vazquez}) implies that $v>0$ in $\Omega$. This shows that $u<\beta$ in $\overline{\Omega}$, contradicting $\|u\|_\co = \beta$.
\end{proof}

The extension of this result to problem~\eqref{eq:PbCV} is not immediate. Taking advantage of Proposition~\ref{prop:equiv}, the strategy is to relate $f$ and $\tilde f$ (defined in~\eqref{eq:def_tilde_f}) through a condition like~\eqref{eq:hyp_smp_2}, and then apply Lemma~\ref{lem:smp} to problem~\eqref{eq:PbSemiCV}.

\begin{proposition}
Let $f, g\in C(\R)$. If $f$ satisfies~\eqref{eq:hyp_smp}, then there is no nonnegative solution $u\in \cuo$ to~\eqref{eq:PbCV} with $\|u\|_\co = \beta$.
\end{proposition}

\begin{remark}
When $1<p\leq 2$, condition~\eqref{eq:hyp_smp} holds whenever $f(\beta)=0$ and $f$ is Lipschitz continuous in $[\beta-\delta,\beta]$ for some $\delta>0$.
\end{remark}

\begin{proof}
By Proposition~\ref{prop:equiv}, we are done if we prove that problem~\eqref{eq:PbSemiCV} has no nonnegative solution ${v\in\cuo}$ with $\|v\|_\co = \Psi(\beta)$. Thanks to Lemma~\ref{lem:smp}, it suffices to show that
\begin{equation}
    \label{eq:Pf_P_SMP_1}
    \limsup_{s\to \Psi(\beta)^-} \frac{\tilde f(s)}{(\Psi(\beta)-s)^{p-1}} <\infty.
\end{equation}
Taking into account the definition of $\tilde f$ (see~\eqref{eq:def_tilde_f}) and performing a change of variables, this limit can be rewritten as
\begin{equation}
    \label{eq:Pf_P_SMP_2}
    \limsup_{s\to \Psi(\beta)^-} \frac{\tilde f(s)}{(\Psi(\beta)-s)^{p-1}} = \limsup_{s\to \beta^-} \frac{f(s)e^{-G(s)}}{(\Psi(\beta)-\Psi(s))^{p-1}} = e^{-G(\beta)} \limsup_{s\to \beta^-} \frac{f(s)}{(\Psi(\beta)-\Psi(s))^{p-1}}.
\end{equation}
Now, we use Taylor's Theorem to write
\begin{equation*}
    \Psi(s) = \Psi(\beta) - \Psi'(\beta) (\beta-s) + o(\beta-s) = \Psi(\beta) - e^{-\frac{1}{p-1}G(\beta)} (\beta-s) + o(\beta-s), \ \forall s\in [\beta-\delta,\beta],
\end{equation*}
where $\delta>0$ and $o(\beta-s)$ is such that $\frac{o(\beta-s)}{\beta-s} \to 0$ as $s\to \beta$. Thus, we obtain that
\begin{equation}
    \label{eq:Pf_P_SMP_3}
    \begin{split}
    \limsup_{s\to \beta^-} \frac{f(s)}{(\Psi(\beta)-\Psi(s))^{p-1}} &= \limsup_{s\to \beta^-} \frac{f(s)}{\left( e^{-\frac{1}{p-1}G(\beta)} (\beta-s) - o(\beta-s) \right)^{p-1}}\\
    & = \limsup_{s\to \beta^-} \frac{f(s)}{(\beta-s)^{p-1}} \left( \frac{\beta-s}{e^{-\frac{1}{p-1}G(\beta)} (\beta-s) - o(\beta-s)} \right)^{p-1} \\[2mm]
    &= e^{G(\beta)}\limsup_{s\to \beta^-} \frac{f(s)}{(\beta-s)^{p-1}}.
    \end{split}
\end{equation}
Joining~\eqref{eq:Pf_P_SMP_2} and~\eqref{eq:Pf_P_SMP_3}, we deduce that
\[
\limsup_{s\to \Psi(\beta)^-} \frac{\tilde f(s)}{(\Psi(\beta)-s)^{p-1}} = \limsup_{s\to \beta^-} \frac{f(s)}{(\beta-s)^{p-1}}.
\]
Since~\eqref{eq:hyp_smp_2} holds (because it is equivalent to~\eqref{eq:hyp_smp}), it follows that~\eqref{eq:Pf_P_SMP_1} also holds, as desired.
\end{proof}

Finally, we note that when $f(s)\approx C(\beta-s)^{k-1}$ near $s=\beta$ for some $0<k<p-1$, the strong maximum principle does not apply and solutions to~\eqref{eq:PbCV} having a flat core may appear. These are solutions ${u\in\cuo}$ such that $\{x\in\Omega: u(x)=\beta\}\neq \emptyset$. We refer the reader to~\cite{Sabina, Guo2} for a more detailed explanation of this phenomenon.

\subsection{Regularizing effect as \texorpdfstring{$p\to 1^+$}{p goes to 1}}

Some of the results proved in Section~\ref{sec:regularizing} remain valid if we fix $L$ and we tend $p$ to 1. Observe that the functions $\Psi_L(s)$ (defined in~\eqref{eq:def_Psi_L}) and $H_{\gamma_1,\gamma_2}(s,L)$ (defined in~\eqref{eq:def_Hgamma}) can be rewritten as
\begin{equation*}
     \Psi_{p,L}(s) \coloneqq \int_0^s e^{-C_1(p,L) G(\eta)} \ \mathrm{d}\eta, \ \forall s\in \R,
\end{equation*}
and
\begin{equation*}
    H_{\gamma_1,\gamma_2}(s,p,L) \coloneqq \int_{s}^{\gamma_2} f(\eta) e^{-C_2(p,L) G(\eta)} \ \mathrm d \eta, \ \forall (s,L) \in [0,\gamma_1]\times \R,
\end{equation*}
where $C_1(p,L) = \frac{1}{p-1}L$ and $C_2(p,L) = \frac{p}{p-1}L$. The key is that if $L<0$, then $C_1(p,L)$ and $C_2(p,L)$ tend to $-\infty$ as $p\to 1^+$; whereas if $L>0$, then $C_1(p,L)$ and $C_2(p,L)$ tend to $\infty$ as $p\to 1^+$. This allows us to reproduce some of the arguments of Section~\ref{sec:regularizing}.

If we fix $L=-1$, then the following regularizing effect appears. The proof is analogous to that of Theorem~\ref{th:conv_beta}.

\begin{theorem}
Let $f\in C(\R)$ satisfy $f(0)\geq 0$ and $f(s)>0$ for $s\in (\alpha,\beta)$, where $0<\alpha<\beta$ are two zeros of $f$, and let $g\in C(\R)$ be positive. Given $\lambda>0$, there is some $p_\lambda>1$ such that the maximal solution $\overline u_p$ in the interval $[0,\beta]$ to problem
\begin{equation*}
\left\{ \begin{alignedat}{2}
-\Delta_p u &= g(u) |\nabla u|^p + \lambda f(u) \quad &&\mbox{in} \;\; \Omega, \\
u &= 0 \quad &&\mbox{on} \;\; \partial\Omega, 
\end{alignedat}
\right.
\end{equation*}
exists for every $p\in (1,p_\lambda]$, and
\[
\|\overline{u}_p\|_{\co} \to \beta \text{ as } p\to 1^+.
\]
\end{theorem}

On the other hand, when $L=1$ and $f$ changes sign in $[0,\beta]$, tending $p$ to 1 causes the area condition~\eqref{eq:area_cond} to be lost. Reasoning as in Proposition~\ref{prop:rangeL}, we obtain the following.

\begin{theorem}
Let $f\in C(\R)$ satisfy $f(0)\geq 0$ and $f(s)>0$ for $s\in (\alpha,\beta)$, where $0<\alpha<\beta$ are two zeros of $f$, and let $g\in C(\R)$ be positive. If $f$ changes sign in $[0,\beta]$, there is some $\tilde p>1$ such that problem
\begin{equation*}
\left\{ \begin{alignedat}{2}
-\Delta_p u + g(u) |\nabla u|^p &= \lambda f(u) \quad &&\mbox{in} \;\; \Omega, \\
u &= 0 \quad &&\mbox{on} \;\; \partial\Omega, 
\end{alignedat}
\right.
\end{equation*}
has no nonnegative solution with maximum in $[\alpha,\beta]$ for any $p\in (1,\tilde p]$ and any $\lambda>0$.
\end{theorem}

\section*{Acknowledgements}
This work has been supported by the Junta de Andaluc\'ia (Spanish research group FQM-424) and by CDTIME. The second author is member of the Gruppo Nazionale per l'Analisi Ma\-te\-ma\-ti\-ca, la Probabilit\`a e le loro Applicazioni (GNAMPA) of the Istituto Nazionale di Alta Matematica (INdAM). Additionally, the third author has been supported by the FPU predoctoral fellowship of the Spanish Ministry of Universities (FPU21/04849).

\end{document}